\documentclass[11pt]{article}
\usepackage{graphicx}
\usepackage{amsmath,amsthm,amssymb,enumerate}%, esint}
\usepackage{euscript,mathrsfs}
\usepackage[left=2cm,right=2cm,top=3.5cm,bottom=3.5cm]{geometry}
\usepackage{color}

\usepackage{soul}

\catcode`\@=11 \@addtoreset{equation}{section}

\catcode`\@=12

\allowdisplaybreaks

\newtheorem{Theorem}{Theorem}[section]
\newtheorem{Proposition}[Theorem]{Proposition}
\newtheorem{Lemma}[Theorem]{Lemma}
\newtheorem{Corollary}[Theorem]{Corollary}

\theoremstyle{definition}
\newtheorem{Definition}[Theorem]{Definition}

\newtheorem{Remark}[Theorem]{Remark}

\newcommand{\bTheorem}[1]{
\begin{Theorem} \label{T#1} }
\newcommand{\eT}{\end{Theorem}}

\newcommand{\bProposition}[1]{
\begin{Proposition} \label{P#1}}
\newcommand{\eP}{\end{Proposition}}

\newcommand{\bLemma}[1]{
\begin{Lemma} \label{L#1} }
\newcommand{\eL}{\end{Lemma}}

\newcommand{\bCorollary}[1]{
\begin{Corollary} \label{C#1} }
\newcommand{\eC}{\end{Corollary}}

\newcommand{\bRemark}[1]{
\begin{Remark} \label{R#1} }
\newcommand{\eR}{\end{Remark}}

\newcommand{\bDefinition}[1]{
\begin{Definition} \label{D#1} }
\newcommand{\eD}{\end{Definition}}

\newcommand{\Ds}{\mathbb{D}_x}

\newcommand{\Ecd}{\mathcal{E}_{{\rm cg}}}

\newcommand{\tvm}{\widetilde{\vc{m}}}
\newcommand{\bfphi}{\boldsymbol{\varphi}}

\newcommand{\bFormula}[1]{
\begin{equation} \label{#1}}
\newcommand{\eF}{\end{equation}}

\newcommand{\Ov}[1]{\overline{#1}}

\newcommand{\vr}{\varrho}

\newcommand{\tvr}{\tilde \vr}
\newcommand{\tvu}{{\tilde \vu}}

\newcommand{\vu}{\vc{u}}
\newcommand{\vm}{\vc{m}}

\newcommand{\vn}{\vc{n}}

\newcommand{\vc}[1]{{\bf #1}}

\newcommand{\Div}{{\rm div}_x}
\newcommand{\Grad}{\nabla_x}

\newcommand{\dx}{\,{\rm d} {x}}

\newcommand{\dt}{\,{\rm d} t }

\newcommand{\intO}[1]{\int_{\Omega} #1 \ \dx}

\newcommand{\D}{{\rm d}}

\newcommand{\ep}{\varepsilon}

\def\softd{{\leavevmode\setbox1=\hbox{d}%
          \hbox to 1.05\wd1{d\kern-0.4ex{\char039}\hss}}}
\definecolor{Cgrey}{rgb}{0.85,0.85,0.85}
\definecolor{Cblue}{rgb}{0.50,0.85,0.85}
\definecolor{Cred}{rgb}{1,0,0}
\definecolor{fancy}{rgb}{0.10,0.85,0.10}

\newcommand\Cbox[2]{%
    \newbox\contentbox%
    \newbox\bkgdbox%
    \setbox\contentbox\hbox to \hsize{%
        \vtop{
            \kern\columnsep
            \hbox to \hsize{%
                \kern\columnsep%
                \advance\hsize by -2\columnsep%
                \setlength{\textwidth}{\hsize}%
                \vbox{
                    \parskip=\baselineskip
                    \parindent=0bp
                    #2
                }%
                \kern\columnsep%
            }%
            \kern\columnsep%
        }%
    }%
    \setbox\bkgdbox\vbox{
        \color{#1}
        \hrule width  \wd\contentbox %
               height \ht\contentbox %
               depth  \dp\contentbox
        \color{black}
    }%
    \wd\bkgdbox=0bp%
    \vbox{\hbox to \hsize{\box\bkgdbox\box\contentbox}}%
    \vskip\baselineskip%
}

\newcommand{\mbb}{\mathbb}

\newcommand{\mc}{\mathcal}

\newcommand{\wtilde}{\widetilde}

\newcommand{\oline}{\overline}

\newcommand{\ra}{\rightarrow}

\newcommand{\g}{\gamma}

\newcommand{\de}{\delta}

\newcommand{\N}{\mathbb{N}}

\allowdisplaybreaks

\def\d{\partial}

\date{\small \today}

\begin{document}

%%%%%%%%%%%%%%%%%%%%%%%%%%%%%%%%

\title{Statistical solutions to the barotropic Navier--Stokes system}

\author{Francesco Fanelli
\thanks{The work of F.F. has been partially supported by the LABEX MILYON (ANR-10-LABX-0070) of Universit\'e de Lyon, within the program ``Investissement d'Avenir''
(ANR-11-IDEX-0007),  and by the projects BORDS (ANR-16-CE40-0027-01) and SingFlows (ANR-18-CE40-0027), all operated by the French National Research Agency (ANR).}
\and Eduard Feireisl
\thanks{The work of E.F. was partially supported by the
Czech Sciences Foundation (GA\v CR), Grant Agreement
18--05974S. The Institute of Mathematics of the Academy of Sciences of
the Czech Republic is supported by RVO:67985840.}
}

%\date{\today}

\maketitle

\centerline{Institut Camille Jordan UMR CNRS 5208, Universit\'e Claude Bernard Lyon 1;}
\centerline{43, Boulevard du 11 novembre 1918, F-69622 Villeurbanne, France}
\centerline{\texttt{fanelli@math.univ-lyon1.fr}}

\bigbreak
\bigbreak
\centerline{Institute of Mathematics of the Academy of Sciences of the Czech Republic;}
\centerline{\v Zitn\' a 25, CZ-115 67 Praha 1, Czech Republic}
\medbreak
\centerline{Institute of Mathematics, Technische Universit\"{a}t Berlin,}
\centerline{Stra{\ss}e des 17. Juni 136, 10623 Berlin, Germany}
\centerline{\texttt{feireisl@math.cas.cz}}

\begin{abstract}

We introduce a new concept of \emph{statistical solution} in the framework of weak solutions to the barotropic Navier--Stokes system with
inhomogeneous boundary conditions. Statistical solution is a family $\{ M_t \}_{t \geq 0}$ of Markov operators on the set of probability measures $\mathfrak{P}[\mathcal{D}]$ on the data space $\mathcal{D}$
containing the initial data $[\vr_0, \vm_0]$ and the boundary data $\vc{d}_B$.
\begin{itemize}
\item
$\{ M_t \}_{t \geq 0}$ possesses
a.a. semigroup property,
\[
M_{t + s}(\nu) = M_t \circ M_s(\nu) \ \mbox{for any}\ t \geq 0, \ \mbox{a.a.}\ s \geq 0, \ \mbox{and any}\ \nu \in \mathfrak{P}[\mathcal{D}].
\]
\item
$\{ M_t \}_{t \geq 0}$ is deterministic when restricted to deterministic data, specifically
\[
M_t( \delta_{[\vr_0, \vm_0, \vc{d}_B]}) = \delta_{[\vr(t, \cdot), \vm(t, \cdot), \vc{d}_B]},\ t \geq 0,
\]
where $[\vr, \vm]$ is a finite energy weak solution of the Navier--Stokes system
corresponding to the data $[\vr_0, \vm_0, \vc{d}_B] \in \mathcal{D}$.
\item
$M_t: \mathfrak{P}[\mathcal{D}] \to \mathfrak{P}[\mathcal{D}]$ is continuous in a suitable Bregman--Wasserstein metric at measures supported by the data giving rise to regular solutions.

\end{itemize}

\end{abstract}


\medbreak
\bigskip

\section{Introduction} \label{s:intro}

In view of the large number of recent results concerning ill--posedness of some iconic problems in fluid mechanics,
see the survey by Buckmaster and Vicol \cite{BucVic} or
\cite{BucVic1}, the \emph{statistical solutions}, introduced in the pioneering work of Foias \cite{Foias},
and Vishik and Fursikov \cite{VisFur} and revisited recently by Foias et al. \cite{BiFoMoTit}, \cite{FMRT}, \cite{FoMoTi}, \cite{FoRoTe3},
\cite{FoRoTe2}, \cite{FoRoTe1}, reappeared as a possible alternative to establish well--posedness in a broader
sense. Flandoli and Romito \cite{FlaRom} exploited this idea proving the existence of Markov selection for the stochastically driven
incompressible Navier--Stokes system. Indeed, in the absence of the white noise forcing, the Markov selection obtained in \cite{FlaRom} may be viewed as a statistical solution of the problem with random initial data in the sense of Vishik and Fursikov \cite{VisFur}. Recently, the method have been adapted to compressible fluid flows in \cite{BrFeHo2018C}. Alternatively, Cardona and Kapitanski \cite{CorKap} modified the method to
handle the deterministic problems obtaining a measurable semiflow selection for a rather general class of evolutionary problems.
Other aspects and problems related to statistical solutions to the incompressible Navier--Stokes system have been studied
by Constantin and Wu \cite{ConWu}, Levant, Ramos, and Titi \cite{LeRaTi}, among others. For recent applications to conservation laws and Euler equations
see Fjordholm, Lanthaler and Mishra \cite{Fjord-L-M}, and Fjordholm and Wiedemann \cite{Fjord-W}.

Our goal is to develop a new concept of \emph{statistical solution} in the context of compressible (barotropic) Newtonian fluids. As the existence of global in time smooth solutions is not known (but still not excluded), the theory is based on the concept of weak solution
in the spirit of Lions \cite{LI4} and \cite{EF70}. The leading idea is that the statistical solutions should share some properties
typical for the well--posed problems:

\begin{itemize}

\item Statistical solutions are interpreted as a Markov semigroup $\{ M_t \}_{t \geq 0}$ of operators assigning to the initial distribution of the data the distribution of the solution at any time $t > 0$.

\item The distribution measure should be supported by smooth solutions as soon as they exist.

\item If the initial distribution is given by a Dirac mass in the data space, its time evolution is a Dirac mass supported by a
weak solution of the problem. In particular, the framework of weak solutions is included as a special case.

\end{itemize}

Last but not least, we show certain stability of strong solutions in the spirit of the weak--strong uniqueness principle known for the deterministic solutions.
To this end, we introduce a new concept of \emph{Bregman--Wasserstein distance} in the spirit of Guo et al. \cite{GuHoLiYa} using the relative energy as cost functional.

\subsection{Problem formulation}
\label{i}

Let $\Omega$ be a bounded domain in $R^d$ (for $d=2,3$). The motion of a compressible viscous fluid contained in $\Omega$
is described by the mass density $\vr = \vr(t,x)$, and the velocity $\vu(t,x)$, $t > 0$, $x \in \Omega$. With the thermal
effects neglected, and with the linear dependence of the viscous stress $\mathbb{S}$ on the velocity gradient $\Grad \vu$,
the time evolution of the fluid is described by the \emph{barotropic Navier--Stokes system}:
\begin{equation} \label{i1}
\begin{split}
\partial_t \vr + \Div (\vr \vu) &= 0,\\
\partial_t (\vr \vu) + \Div (\vr \vu \otimes \vu) + \Grad p(\vr) &=
\Div \mathbb{S}(\Ds \vu) + \vr \vc{g},
\end{split}
\end{equation}
with the viscous stress tensor
\[
\mathbb{S}(\Ds \vu) = \mu \left( \Grad \vu + \Grad^t \vu - \frac{2}{d} \Div \vu \mathbb{I} \right) +
\lambda \Div \vu \mathbb{I},\ \mbox{where}\ \mu > 0, \ \lambda \geq 0,\ \mbox{with}\ \Ds \vu \equiv \frac{1}{2} \Big( \Grad \vu + \Grad^t \vu \Big).
\]

We impose the physically relevant in/out flow boundary conditions for the velocity:
\begin{equation} \label{i2}
\vu|_{\partial \Omega} = \vu_B .
\end{equation}
Accordingly, we can decompose the boundary $\d\Omega$ as
\[
\partial \Omega = \Gamma_{\rm in} \cup \Gamma_{\rm out}
\]
where
\[
\Gamma_{\rm in} := \left\{ x \in \partial \Omega \ \Big|\
\ \mbox{the outer normal}\ \vc{n}(x) \ \mbox{exists, and}\ \vu_B(x) \cdot \vc{n}(x) < 0 \right\}.
\]
In addition, the density is given on the inflow part of the boundary:
\begin{equation} \label{i3}
\vr|_{\Gamma_{\rm in}} = \vr_B.
\end{equation}
The initial state is given by the initial conditions:
\begin{equation} \label{i4}
\vr(0, \cdot) = \vr_0, \ \ \big(\vr \vu\big) (0,\cdot) = \vm_0\,.
\end{equation}

We refer to
the quantity $[\vr_0, \vm_0, \vr_B, \vu_B, \vc{g}]$ as (given) \emph{data}. Ideally, the solution $[\vr, \vu]$ is determined uniquely by the data, however, the recent state--of--the--art of the mathematical theory does not provide a positive answer to this conjecture,
except for short time intervals and/or smooth or small data in some particular cases, see Bothe and Pr\"uss \cite{BotPru}, Matsumura and Nishida
\cite{MANI1}, \cite{MANI}, Valli and Zajaczkowski \cite{VAZA}, among others. A suitable platform for studying the global properties of the
system \eqref{i1}--\eqref{i4} is the theory of \emph{weak solutions}, where the existence of global--in--time solutions have been established recently by Chang, Jin, and Novotn\' y \cite{ChJiNo}. In view of the recent results obtained via the method of convex integration, see Buckmaster et al. \cite{BucVic1}, \cite{BuDeSzVi}, a proper concept of weak solution must be accompanied by the associated \emph{energy balance} that guarantees in particular stability of strong solutions in the class of weak solutions (weak--strong uniqueness principle), cf. Brenier et al. \cite{BrDeSz}, Germain \cite{Ger}, and, more recently \cite{AbbFeiNov}.

The energy balance associated to the system \eqref{i1}--\eqref{i4} reads
\begin{equation} \label{i5}
\begin{split}
\frac{\D}{\dt} &\intO{\left[ \frac{1}{2} \vr \left|\vu - \vu_B\right|^2 + P(\vr) \right] }  +
\intO{ \mathbb{S}(\Ds \vu) : \Ds \vu }  \\
&+ \int_{\Gamma_{\rm out}} P(\vr)  \vu_B \cdot \vc{n} \ \D S_x   +
\int_{\Gamma_{\rm in}} P(\vr_B)  \vu_B \cdot \vc{n} \ \D S_x
\\	
&\leq -
\intO{ \left[ \vr \vu \otimes \vu + p(\vr) \mathbb{I} \right]  :  \Grad \vu_B }  - \intO{ {\vr} \vu  \cdot \vu_B \cdot \Grad \vu_B  }
\dt\\ &\ \ \  + \intO{ \mathbb{S}(\Ds \vu) : \Ds \vu_B } + \intO{ \vr \vc{g} \cdot (\vu - \vu_B) },
\end{split}
\end{equation}
where $P(\vr)$ is the pressure potential determined modulo a linear function of $\vr$ by the identity
\[
P'(\vr) \vr - P(\vr) = p(\vr).
\]
The inequality in \eqref{i5} is usually attributed to possible ``anomalous'' energy dissipation inherent to \emph{weak} solutions.
The total energy
\[
\mathcal{E} \equiv  \intO{\left[ \frac{1}{2} \vr \left|\vu - \vu_B\right|^2 + P(\vr) \right] },
\ \mbox{or rather its c\` agl\` ad version}\ \Ecd,
\]
considered as an auxilliary state variable will play a crucial role in the forthcoming analysis.

\begin{Remark} \label{MR3}

The energy can be defined in terms of the density and momentum $\vm\equiv \vr\vu$ that are weakly continuous quantities in time:
\[
E \left(\vr, \vm \ \Big| \vu_B \right) \equiv \left[ \frac{1}{2} \vr |\vu - \vu_B|^2 + P(\vr) \right] =
\left[ \frac{1}{2} \frac{|\vm|^2}{\vr} - \vm \cdot \vu_B + \frac{1}{2} \vr |\vu_b|^2 + P(\vr) \right].
\]
Moreover, with the convention
\[
E \left(\vm, \vu \ \Big| \vu_b \right) = \infty \ \mbox{if}\ \vr < 0 \ \mbox{or}\ \vr = 0, \vm \ne 0,\
E \left(\vm, \vu \ \Big| \vu_b \right) = 0 \ \mbox{if}\ \vr = 0, \ \vm = 0,
\]
$E$ is a convex l.s.c. function of $[\vr, \vm] \in R^{d + 1}$.

\end{Remark}

\subsection{Semiflow selection and pushforward measure}

Statistical solution in the spirit of Foias et al. \cite{FoRoTe3} is a family of probability measures $\{ \mathcal{V}_t \}_{t \geq 0}$
defined on the state space associated to the solution $[\vr, \vu]$. As a matter of fact, it is more convenient to consider
the density $\vr = \vr(t,x)$ and the momentum $\vm(t,x) = (\vr \vu)(t,x)$ as the state variables, as they are weakly continuous
and therefore well defined as functions of the time variable.

Inspired by Cardona and Kapitanski \cite{CorKap}, we adopt the selection procedure proposed by Krylov \cite{KrylNV} to identify a semiflow selection assigning to given data
$[\vr_0, \vm_0, \vr_B, \vu_B, \vc{g}]$
a single trajectory
\[
t \in [0, \infty) \mapsto [\vr(t, \cdot), \vm(t, \cdot)](\vr_0, \vm_0, \vr_B, \vu_B, \vc{g})
\]
in a (Borel) measurable way. To make the notation more concise, we denote the boundary data as
\[
\vc{d}_B \equiv [\vr_B, \vu_B, \vc{g}].
\]
The idea is then to
define $\{ \mathcal{V}_t \}_{t > 0}$ via the associated pushforward measure, specifically,
\begin{equation} \label{i6}
\int \Phi (\vr, \vm; \vc{d}_B) \D \mathcal{V}_t (\vr, \vm, \vc{d}_B) =
\int \Phi \Big( [\vr, \vm] (t, \cdot) (\vr_0, \vm_0, \vc{d}_B), \vc{d}_B \Big) \D \mathcal{V}_0 (\vr_0, \vm_0, \vc{d}_B)
\end{equation}
for any bounded continuous function $\Phi$ defined on the phase space
associated to $[\vr, \vm, \vc{d}_B]$,
and a given measure $\mathcal{V}_0$ defined on the data space. The mapping
\[
M_t : \mathcal{M}^+ \to \mathcal{M}^+,\ M_t [\mathcal{V}_0] = \mathcal{V}_t,\ t \geq 0,
\]
represents a statistical solution of the Navier--Stokes system.

Unfortunately, such a procedure cannot be applied in a straightforward manner as the total energy $\mathcal{E}$ must be included in the state variables. This requires replacing $\mathcal{E}(t)$
by its c\` agl\` ad version $\Ecd$ - a BV function
determined through the energy inequality \eqref{i5} - whereas the identity
\[
\Ecd(t) = \mathcal{E}(t)
\equiv \intO{ \left[ \frac{1}{2} \vr \left|\frac{\vm}{\vr} - \vu_B \right|^2 + P(\vr) \right](t, \cdot) }
\]
holds only for a.a. $t \in [0, \infty)$. As a result, $\Ecd$ must be considered as ``independent'' state variable ranging
in a trajectory space that accommodates its pointwise values in time. We adopt a variant of the Skorokhod topology
${D}([0, \infty); R)$ as a trajectory space for the energy, see Jakubowski \cite{Jaku}. Finally, we remark that a similar procedure
in the context of the barotropic Navier--Stokes system with homogeneous boundary conditions has been performed by Basari\v c \cite{Basa1},
where the phase space for the energy is taken $L^1_{\rm loc}[0, \infty)$. The Skorokhod topology seems more convenient as the pointwise in time values of the total energy are well defined while they correspond merely to the Lebesgue points in the $L^1_{\rm loc}$ setting.

Following the above delineated strategy, we first identify a semiflow selection that assigns to any data
$(\vr_0, \vm_0, \vc{d}_B)$ and $\mathcal{E}_0$ a unique solution trajectory
\[
\begin{split}
U: \Big( t; (\vr_0, \vm_0, \mathcal{E}_0, \vc{d}_B) \Big) &\mapsto [\vr(t, \cdot), \vm(t,\cdot), \Ecd(t), \vc{d}_B],\\
U \Big(t + s; (\vr_0, \vm_0, \mathcal{E}_0, \vc{d}_B ) \Big) &=
U \left( t;  U \Big(s;(\vr_0, \vm_0, \mathcal{E}_0, \vc{d}_B ) \Big) \right) \ \mbox{for any}\ t,s \geq 0.
\end{split}
\]
The associated pushforward measure can be defined analogously to \eqref{i6},
\[
\int \Phi (\vr, \vm, \Ecd, \vc{d}_B) \D \mathcal{V}_t (\vr, \vm, \Ecd, \vc{d}_B) =
\int \Phi \Big( [\vr, \vm, \Ecd] (t, \cdot) (\vr_0, \vm_0, \mathcal{E}_0, \vc{d}_B), \vc{d}_B \Big) \D \mathcal{V}_0 (\vr_0, \vm_0, \mathcal{E}_0, \vc{d}_B)
\]
for any bounded Borel measurable function $\Phi$ on the data space. Note that
the trajectories consist of \emph{three} components - $[\vr, \vm]$, and $\mathcal{E}$.  Finally, seeing that
\[
\Ecd(t) =
\intO{ \left[ \frac{1}{2} \vr \left|\frac{\vm}{\vr} - \vu_B \right|^2 + P(\vr) \right](t, \cdot) }
\ \mbox{for a.a.}\ t \in [0, \infty)
\]
we deduce the associated energy balance for \emph{a.a.} $t \in [0, \infty)$ similarly to Foias et al. \cite{FoRoTe3}, or
Vishik and Fursikov \cite{VisFur}. Accordingly, the associated family of Markov operators $\{ M_t \}_{t \geq 0}$
that assigns $\mathcal{V}_t$ to $\mathcal{V}_0$ will enjoy the semigroup property only a.a., specifically,
\[
M_{t + s}(\mc V) = M_t \circ M_s (\mc V) \ \mbox{for all}\ t \ \mbox{and a.a.}\ s \ \mbox{including}\ s = 0,
\]
where the exceptional set of times $s$ depends on the measure $\mc V$.

The paper is organized as follows:

\begin{itemize}

\item In Section \ref{M}, we recall the necessary material concerning the finite energy weak solutions to the barotropic Navier--Stokes system.

\item In Section \ref{SM}, we introduce the concept of
statistical solutions and state our main result.

\item Section \ref{ET} is devoted to the existence theory for the barotropic Navier--Stokes system. Although the basic existence theorem
is proved by Chang, Jin, and Novotn\' y \cite{ChJiNo} (cf. also Lions \cite[Chapter 7, Section 7.6]{LI4}), our definition of weak solution is slightly different and requires certain modifications in the proofs.

\item In Section \ref{S}, we show the existence of a semiflow selection, in particular the continuity of the energy in the
Skorokhod space $D([0, \infty); R)$.

\item Section \ref{s:ST} contains the proof of the main results.

\item The paper is concluded by a short discussion in Section \ref{C}. In particular, we show that the statististical solutions are continuous at measures supported by regular data that may be seen as an analogue of the weak--strong uniqueness property for deterministic solutions.

\end{itemize}

\section{Weak solutions to the Navier--Stokes system}
\label{M}

We start by specifying the hypotheses on the equation of state (EOS for brevity), which links the pressure to the density.
We consider a pressure $p$ such that
\begin{equation}  \label{MH2}
\begin{split}
&p \in C^1[0, \infty),\ p(0) = 0, \\ %\ p'(\vr) > 0 \ \mbox{for}\ \vr > 0, \ p'(\vr) \approx \vr^{\gamma - 1}, \ \gamma > 1 \ \mbox{as}\ \vr \to \infty,
& \ p'(\vr) > 0 \ \mbox{for}\ \vr > 0\qquad \mbox{ and }\qquad \underline{p} \vr^{\gamma - 1} \leq p'(\vr) \leq \Ov{p} \vr^{\gamma - 1}  \ \mbox{for all}\ \vr > 1, % \ \mbox{where}\ \underline{p} > 0.
\end{split}
\end{equation}
for some $\g>1$ and two constants $0<\underline{p}\leq\oline p$. The associated pressure potential $P$,
\[
P'(\vr) \vr - P(\vr) = p(\vr),\ \ \mbox{is normalized by setting}\ P(0) = 0.
\]
In particular,
\[
P''(\vr) = \frac{p'(\vr)}{\vr} > 0 \ \mbox{for}\ \vr > 0\ \Rightarrow \
P \ \mbox{is (strictly) convex.}
\]
Accordingly, we
may assume that either
\[
P'(\vr) \to - \infty \ \mbox{if} \ \vr \to 0^+\,,\qquad \ \mbox{ or }\qquad \ P'(\vr) \to 0 \ \mbox{if}\ \vr \to 0^+,
\]
adding a linear function to $P$ in the latter case, if necessary.

To avoid technical difficulties, we suppose that $\partial \Omega$ is smooth of class $C^2$. Similarly, we consider
\begin{equation} \label{MH3}
\vu_B = \vu_B|_{\partial \Omega},\ \vu_B \in C^1(\Ov{\Omega}; R^d),\
\vr_B \in C(\partial {\Omega}),\ \vr_B \geq \underline{\vr} > 0,
\end{equation}
where $\underline{\vr}$ is a positive constant. These hypotheses allow us to use the available existence theory developed in
Chang, Jin, and Novotn\' y \cite{ChJiNo}.
As a matter of fact, they could be considerably relaxed in the spirit of \cite{AbbFeiNov}.

\subsection{Weak solutions, energy inequality}
\label{MS1}

Having collected the necessary preliminary material, we introduce the concept of \emph{finite energy weak solution}.

\begin{Definition} [Finite energy weak solution] \label{MD1}

Let $\Omega \subset R^d$, $d=2,3$ be a bounded domain of class $C^2$. Let $[\vr_B, \vu_B]$ belong to the class \eqref{MH3},
and let
\[
\vc{g} \in C(\Ov{\Omega}; R^d).
\]

We say that $[\vr, \vm]$ is a \emph{finite energy weak solution} of the problem \eqref{i1}--\eqref{i4}
with the total energy $\Ecd$ in $[0, \infty)
\times \Omega$, and the initial data $[\vr_0, \vm_0, \mathcal{E}_0]$, if the following is satisfied:

\begin{itemize}

\item {\bf Regularity class.}
\[
\begin{split}
\vr &\in C_{{\rm weak,loc}}([0, \infty); L^\gamma (\Omega)) \cap
L^\gamma_{\rm loc}(0,T; L^\gamma (\partial \Omega, |\vu_B \cdot \vc{n}| \dx )),\ \vr \geq 0,\\
\vm &\in C_{{\rm weak,loc}}([0, \infty); L^{\frac{2 \gamma}{\gamma + 1}}(\Omega; R^d)),\\
\vm &= \vr \vu \ \mbox{a.a.},\ \ \mbox{where }\ (\vu - \vu_B) \in L^2_{\rm loc}([0,\infty); W^{1,2}_0(\Omega; R^d)),\\
\Ecd &\in BV_{\rm loc}[0, \infty), \ \mbox{c\` agl\` ag},\ \Ecd (0-) \equiv \mathcal{E}_0.
\end{split}
\]

\item
{\bf Equation of continuity.}

\begin{equation} \label{M1}
\begin{split}
- \intO{ \vr_0 \varphi }  &+
\int_0^\infty \int_{\Gamma_{\rm out}} \varphi \vr \vu_B \cdot \vc{n} \ \D \ S_x
+
\int_0^\infty \int_{\Gamma_{\rm in}} \varphi \vr_B \vu_B \cdot \vc{n} \ \D \ S_x\\ &=
\int_0^\infty \intO{ \Big[ \vr \partial_t \varphi + \vr \vu \cdot \Grad \varphi \Big] } \dt,
\end{split}
\end{equation}
holds for any $\varphi \in C^1_c([0,\infty) \times \Ov{\Omega})$. In addition, a renormalized version of \eqref{M1}:
\begin{equation} \label{M1a}
- \intO{ b(\vr_0) \varphi }  =
\int_0^\infty \intO{ \Big[ b(\vr) \partial_t \varphi + b(\vr) \vu \cdot \Grad \varphi -
\Big( b'(\vr) \vr - b(\vr) \Big) \Div \vu \Big] } \dt
\end{equation}
holds for any
$\varphi \in C^1_c([0,\infty) \times {\Omega})$, and any $b \in C^1[0, \infty)$, $b' \in C_c[0, \infty)$.

\item
{\bf Momentum equation.}
\begin{equation} \label{M2}
- \intO{ \vm_0 \cdot \bfphi } =
\int_0^\infty \intO{ \Big[ \vr \vu \cdot \partial_t \bfphi + \vr \vu \otimes \vu : \Grad \bfphi
+ p(\vr) \Div \bfphi - \mathbb{S}(\Ds \vu) : \Grad \bfphi + \vr \vc{g} \cdot \bfphi \Big] }
\end{equation}
holds for any $\bfphi \in C^1_c([0,\infty) \times {\Omega}; R^d)$.

\item {\bf Total energy balance.}

The total energy
\begin{equation} \label{M3a}
\Ecd \in BV_{\rm loc}[0, \infty),\
\Ecd(t) = \intO{ \left[ \frac{1}{2}\vr \left| \frac{\vm}{\vr} - \vu_B \right|^2 + P(\vr) \right](t, \cdot)}
\ \mbox{for a.a.}\ t \in [0, \infty)
\end{equation}
satisfies
\begin{equation} \label{M3}
\begin{split}
&- \int_0^\infty \partial_t \psi \Ecd \dt  +
\int_0^\infty \psi \intO{ \mathbb{S}(\Ds \vu) : \Ds \vu } \dt \\
&+ \int_0^\infty \psi \int_{\Gamma_{\rm out}} P(\vr)  \vu_B \cdot \vc{n} \ \D S_x \dt  +
\int_0^\infty \psi \int_{\Gamma_{\rm in}} P(\vr_B)  \vu_B \cdot \vc{n} \ \D S_x \dt
\\	
&\leq
\psi(0) \mathcal{E}_0  \\
&-
\intO{ \int_0^\infty \psi \left[ \vr \vu \otimes \vu + p(\vr) \mathbb{I} \right]  :  \Grad \vu_B } \dt  -
\int_0^\infty \psi \intO{ {\vr} \vu  \cdot \vu_B \cdot \Grad \vu_B  }\dt
\\ &+ \int_0^\infty \psi \intO{ \mathbb{S}(\Ds \vu) : \Ds \vu_B }\dt  + \int_0^\infty \psi \intO{ \vr \vc{g} \cdot (\vu -
\vu_B) } \dt
\end{split}
\end{equation}
for any
$\psi \in C^1_c[0, \infty)$, $\psi \geq 0$.
\end{itemize}

\end{Definition}

As the functions $\vr$ and $\vm$ are weakly continuous in time, we have
\begin{equation} \label{MR4}
\Ecd(t \pm) \geq \mathcal{E}(t) \equiv \intO{ E \left( \vr, \vm \Big| \vu_B \right) (t, \cdot) }
\ \mbox{for any} \ t \in [0, \infty), \ \mbox{with the convention}\ \mathcal{E}(0-) = \mathcal{E}_0.
\end{equation}
Moreover,
\begin{equation} \label{MR5}
\begin{split}
\Ecd(t \pm) &= \mathcal{E}(t) = \intO{ E \left( \vr, \vm \Big| \vu_B \right) (t, \cdot) }\\
\ \mbox{at any Lebesgue point}\ &t \in (0, \infty) \ \mbox{of the function}\
\mathcal{E} \in L^\infty_{\rm loc}[0, \infty).
\end{split}
\end{equation}

\begin{Remark} \label{RM4}

Note that all the times $t$ for which \eqref{MR5} holds are automatically points of continuity of $\Ecd \in BV_{\rm loc}[0, \infty)$.

\end{Remark}

Definition \ref{MD1} is more in the spirit of Lions \cite{LI4} than Chang, Jin, and Novotn\' y \cite{ChJiNo}. The main point is
the weak formulation of the equation of continuity \eqref{M1} that includes also the ``trace'' of $\vr$ on $\partial \Omega$.
Here, $\vr|_{\partial \Omega}$ is understood as
\begin{equation} \label{trace}
(\vr|_{\partial \Omega}) \vu_B \cdot \vc{n} = \vm \cdot \vc{n} \ \mbox{on}\ \partial \Omega,
\end{equation}
where the traces of both $\vu_B$ and $\vm \cdot \vc{n}$ are well defined, cf. Chen, Torres, Ziemer \cite{ChToZi}. With this convention, we have
\[
\vr|_{\Gamma_{\rm in}} = \vr_B.
\]
We point out that $\vr$ belongs {\it a priori} only to the Lebesgue space $L^\gamma(\Omega)$ for which the boundary trace is not well defined. Relation \eqref{trace} is therefore understood as a definition of the trace of $\vr$ on the set where $\vu_B \cdot \vc{n} \ne 0$.

\section{Main result}
\label{SM}

Some preliminaries are needed before we state our main result. To begin with, it is convenient for the
trajectories $t \mapsto [\vr, \vm](t, \cdot)$ to range in a \emph{Polish space} rather than
$L^\gamma \times L^{\frac{2 \gamma}{\gamma + 1}}$ endowed with the (non--metrizable) weak topology. One possibility is to adapt the
approach of Cardona and Kapitanski \cite{CorKap} based on considering the topology of the injective limit of bounded
(weakly metrizable) balls in $L^\gamma \times L^{\frac{2 \gamma}{\gamma + 1}}$. Here, we opted for a simpler way replacing
$L^\gamma \times L^{\frac{2 \gamma}{\gamma + 1}}$ by a larger space
\[
W^{-k,2}(\Omega) \times W^{-k,2}(\Omega; R^d), \ k > \frac{d}{2} + 1.
\]
Note that
\[
W^{-k,2} (\Omega) = \left[ W^{k,2}_0(\Omega) \right]^* \ \mbox{is a separable Hilbert space},\
W^{k,2}_0 \hookrightarrow\hookrightarrow C^1(\Ov{\Omega}) \ \mbox{if}\ k > \frac{d}{2} + 1.
\]
We adopt the standard identification of $W^{k,2}_0$ as a subspace of $W^{-k,2}$ via the Riesz isometry,
\[
W^{k,2}_0 \hookrightarrow L^2 \approx (L^2)^* \hookrightarrow W^{-k,2}.
\]

As $W^{k,2}_0$ can be always identified with a domain of a suitable elliptic operator, we suppose there is an
$L^2-$orthonormal basis $\{ r_m \}_{m=1}^\infty$ of $W^{k,2}_0(\Omega)$ such that
\begin{equation} \label{Mr1}
\begin{split}
\left< \vr, s \right>_{W^{-k,2}(\Omega)} &= \sum_{m=1}^\infty \lambda_m^{-k/2} \left(\intO{ \vr r_m }\right)
\left(\intO{ s r_m } \right)\\ &\mbox{for a suitable sequence of eigenvalues}\ \lambda_m \to \infty.
\end{split}
\end{equation}
Similarly
\begin{equation} \label{Mr2}
\begin{split}
\left< \vm, \vc{v} \right>_{W^{-k,2}(\Omega; R^d)} &= \sum_{m=1}^\infty \Lambda_m^{-k/2} \left(\intO{ \vm \cdot \vc{w}_m }\right)
\left(\intO{ \vc{v} \cdot \vc{w}_i } \right)\\ &\mbox{for a suitable sequence of eigenvalues}\ \Lambda_m \to \infty,
\end{split}
\end{equation}
where $\{ \vc{w}_m \}_{m=1}^\infty$ is orthonormal in $L^2(\Omega; R^d)$.

Finally, we introduce the projections,
\[
\begin{split}
\vr_M &\equiv
\left[ \intO{ \vr r_1 }, \dots,\intO{ \vr r_M } \right], \vm_M \equiv \left[ \intO{ \vm \cdot \vc{w}_1 }, \dots,\intO{ \vm \cdot \vc{w}_M} \right]
\\ [\vr_M, \vm_M]  &\in W_M \approx R^{2M}.
\end{split}
\]

\subsection{Data space}

The data space must accommodate the initial conditions $[\vr_0, \vm_0]$ as well as
the boundary data $[\vr_B, \vu_B]$ and the driving force $\vc{g}$.
Accordingly, we introduce
\begin{equation} \label{SM2}
\begin{split}
\mathcal{D} = \Big\{ [\vr_0, \vm_0, \vr_B, \vu_B, \vc{g} ] \ \Big|  \ &
\vr_0 \in L^\gamma(\Omega), \vm_0 \in L^{\frac{2 \gamma}{\gamma + 1}}(\Omega; R^d),\ \intO{ E \left( \vr_0, \vm_0 \ \Big| \ \vu_B \right) } < \infty
\\ &\vr_B \in C(\partial {\Omega}),\ \vr_B \geq \underline{\vr} > 0,\ \vu_B \in C^1(\Ov{\Omega}; R^d),\
\vc{g} \in C(\Ov{\Omega};R^d)
 \Big\}
\end{split}
\end{equation}
- a Borel subset of the Polish space
\[
X_{\mathcal{D}} \equiv W^{-k,2}(\Omega) \times W^{-k,2}(\Omega; R^d) \times C({\partial \Omega}) \times
C^1(\Ov{\Omega}; R^d) \times C(\Ov{\Omega}; R^d).
\]
Indeed
\[
\begin{split}
\mathcal{D} = \cup_{N > 0} \Big\{ [\vr_0, \vm_0, \vr_B, \vu_B, \vc{g} ] \ \Big|  \ &
\vr_0 \in L^\gamma (\Omega), \vm_0 \in L^{\frac{2 \gamma}{\gamma + 1}}(\Omega; R^d),\ \intO{ E \left( \vr_0, \vm_0 \ \Big| \ \vu_B \right) } \leq N
\\ &\vr_B \in C(\Ov{\Omega}),\ \vr_B \geq \underline{\vr},\ \vu_B \in C^1(\Ov{\Omega}; R^d),\
\vc{g} \in C(\Ov{\Omega}; R^d)
 \Big\}
\end{split}
\]
where the set on the right--hand side is a countable union of closed subsets of $X_{\mathcal{D}}$.

For notational convenience, from now on we will denote by $\vc{d}_B\in C(\d\Omega)\times C^1(\oline{\Omega};R^d)\times C(\oline{\Omega};R^d)$ the triplet of data $[\vr_B,\vu_B,\vc{g}]$.

\subsection{Statistical solution}

We are ready to introduce the concept of \emph{statistical solution} to the problem \eqref{i1}--\eqref{i4}. The statistical solution reflects the time evolution of the \emph{distribution} of the data and of the solution at later times. Accordingly, the best
way to describe the evolution of the initial data distribution is the semigroup of linear operators on the set of probability measures defined on the data space $\mathcal{D}$. We consider the set $\mathcal{D}$ as
a Borel subset of the space $X_{\mathcal{D}}$ and denote
\[
\mathfrak{P}[ \mathcal{D}] = \left\{ \nu \ \Big| \ \nu \ \mbox{a complete Borel probability measure on} \ X_{\mathcal{D}}, \
{\rm supp}[\nu] \subset \mathcal{D}
\right\}.
\]

\begin{Definition}[Statistical solution] \label{SD1}

A \emph{statistical solution} of the problem \eqref{i1}--\eqref{i4} is a family of (Markov) operators $\{ M_t \}_{t \geq 0}$,
\[
M_t : \mathfrak{P}[\mathcal{D}] \to \mathfrak{P}[\mathcal{D}] \ \mbox{for any}\ t \geq 0,
\]
enjoying the following properties:
\begin{itemize}
\item
\[
M_0 (\nu) = \nu \ \mbox{for any}\ \nu \in \mathfrak{P}[\mathcal{D}];
\]
\item
\[
M_t \left( \sum_{i=1}^N \alpha_i \nu_i \right) =
\sum_{i=1}^N \alpha_i M_t (\nu_i), \ \mbox{for any}\ \alpha_i \geq 0, \ \sum_{i=1}^N \alpha_i = 1,\
\nu_i \in \mathfrak{P}(\mathcal{D}), \ t \geq 0;
\]
\item
\begin{equation} \label{SSM1}
M_{t + s} = M_t \circ M_s \ \mbox{for any}\ t \geq 0 \ \mbox{and a.a.}\ s \in (0, \infty);
\end{equation}
\item
\begin{equation} \label{SSM1a}
t \mapsto M_t (\nu) \ \mbox{is continuous with respect to the weak topology on}\ \mathfrak{P}(\mathcal{D})
\end{equation}
for any $\nu \in \mathfrak{P}(\mathcal{D})$;
\item
\begin{equation} \label{SSM2}
\begin{split}
{M}_t \left( \delta_{[\vr_0, \vm_0, \vc{d}_B]} \right) &= \delta_{[\vr(t, \cdot),
\vm(t, \cdot), \vc{d}_B]} \ \mbox{for any}\ t \geq 0, \\
{M}_t \left( \nu \right) &= \int_{\mathcal{D}} \delta_{[\vr(t, \cdot),
\vm(t, \cdot), \vc{d}_B]} \D \nu (\vr_0, \vm_0, \vc{d}_B)\ \mbox{for any}\ \nu \in \mathfrak{P}[\mathcal{D}],
\end{split}
\end{equation}
where $[\vr, \vm]$ is a finite energy weak solution in the sense of Definition \ref{MD1}, with the data
$(\vr_0, \vm_0, \vc{d}_B)$ and
\[
\mathcal{E}_0 = \intO{ E \left( \vr_0, \vm_0 \ \Big| \vu_B \right) }.
\]
\end{itemize}

\end{Definition}

\begin{Remark} \label{Rcd1}

It follows from \eqref{SSM1a} that
the mapping
\[
M : [0, \infty) \times \mathcal{D} \to \mathcal{D}, \
(t, [\vr_0, \vm_0, \vc{d}_B]) \mapsto  M_t \left( \delta_{[\vr_0, \vm_0, \vc{d}_B]} \right)
= \delta_{[\vr(t, \cdot), \vm(t, \cdot), \vc{d}_B]} \approx
[\vr(t, \cdot), \vm(t, \cdot), \vc{d}_B]
\]
is $\left( [0, \infty) \times \mathcal{D}; \mathcal{D} \right)$ Borel measurable. Here we have identified the
data space $\mathcal{D}$ with a subspace of $\mathfrak{P}[\mathcal{D}]$,
\[
[\vr, \vm, \vc{d}_B] \in \mathcal{D} \approx \delta_{ [\vr, \vm, \vc{d}_B]} \in \mathfrak{P}[\mathcal{D}].
\]

\end{Remark}

\begin{Remark} \label{Rcd2}

The property \eqref{SSM2} anticipates the existence of a Borel measurable mapping on $\mathcal{D}$:
\[
\vc{U}(t): [\vr_0, \vm_0, \vc{d}_B] \mapsto [\vr(t, \cdot), \vm(t,\cdot), \vc{d}_B],
\]
where $[\vr, \vm]$ is a finite energy weak solution corresponding to the data
$(\vr_0, \vm_0, \vc{d}_B)$ and
\[
\mathcal{E}_0 = \intO{ E \left( \vr_0, \vm_0 \ \Big| \vu_B \right) }.
\]
The mapping
\[
\Phi \in BM (\mathcal{D}) \mapsto \Phi \circ \vc{U}(t) \in BM (\mathcal{D})
\]
is called the \emph{adjoint} of the Markov operator $M_t$.

\end{Remark}

The ``almost semigroup'' property \eqref{SSM1} should be understood in the following sense: For any
$\nu \in \mathfrak{P}[ \mathcal{D} ]$ there exists a set of $s \in (0, \infty)$ of full measure such that
\eqref{SSM1} holds. In general, the set of ``exceptional times'' depends on the measure $\nu$.

Thanks to the property \eqref{SSM2}, the statistical solution reduces to a weak solution provided the data are concentrated at one point.
Thus the framework is a proper extension of the standard concept of weak solution, contained as a special case. Statistical solutions share the same deficiency with the weak solutions -- they are not (known to be) uniquely determined by the initial data.
As we shall see in Section \ref{C}, however, uniqueness can be restored on the sets of the data that give rise to smooth solutions of the problem.

\subsection{Main result}

We are ready to state our main result.

\begin{Theorem}\label{ST1}

Let the pressure be given by the EOS \eqref{MH2}, with $\gamma > \frac{d}{2}$.
Let $\mathcal{D}$ be the data set introduced in \eqref{SM2}. Let $\mathcal{V}_0$ be a complete Borel measure on the
(separable) Banach space
\[
X_{\mathcal{D}} = W^{-k,2}(\Omega) \times W^{-k,2}(\Omega; R^d) \times C(\partial{\Omega}) \times C^1(\Ov{\Omega}; R^d) \times
C(\Ov{\Omega}; R^d)
\]
such that
\[
{\rm supp}[\mathcal{V}_0] \subset \mathcal{D}.
\]

Then the following holds:
\begin{itemize}

\item
There exists a mapping
\begin{equation} \label{SSM4}
[\vr, \vm]: [0, \infty) \times \mathcal{D} \to W^{-k,2}(\Omega) \times W^{-k,2}(\Omega; R^d),
\end{equation}
measurable with respect to $\dt \times \mathcal{V}_0$,
enjoying the following properties:
\begin{itemize}
\item
\begin{equation} \label{SSM4a}
[\vr, \vm] (0; \vr_0, \vm_0, \vc{d}_B) =  [\vr_0, \vm_0];
\end{equation}
\item
\begin{equation} \label{SSM4b}
t \mapsto [\vr, \vm] (t; \vr_0, \vm_0, \vc{d}_B) \in C_{\rm loc}([0, \infty); X_{\mathcal{D}})
\end{equation}
is a finite energy weak solution of the Navier--Stokes system \eqref{i1}--\eqref{i4} in the sense of Definition \ref{MD1}, with the data
\[
[\vr_0, \vm_0, \vc{d}_B] \in \mathcal{D}, \ \mbox{and the initial energy}\ \mathcal{E}_0 = \intO{ E\left( \vr_0, \vm_0 \ \Big|\ \vu_B \right) };
\]
\item
\begin{equation} \label{SSM4c}
[\vr, \vm] (t + s; \vr_0, \vm_0, \vc{d}_B) =
[\vr, \vm]\Big(t, [\vr, \vm] (s; \vr_0, \vm_0, \vc{d}_B), \vc{d}_B \Big)
\ \mbox{for any}\ t \geq 0 \ \mbox{and a.a.}\ s \in [0, \infty).
\end{equation}

\end{itemize}

\item

The family $\{ \mathcal{V}_t \}_{t \geq 0}$ of Borel measures on $X_{\mathcal{D}}$ defined
as
\begin{equation} \label{SSM5}
\begin{split}
\int_{X_{\mathcal{D}}} \Phi ( \vr, \vm , \vc{d}_B ) \D \mathcal{V}_t (\vr, \vm, \vc{d}_B) &\equiv
\int_{\mathcal{D}} \Phi \Big(\vr (t; \vr_0, \vm_0, \vc{d}_B) , \vm(t; \vr_0, \vm_0, \vc{d}_B), \vc{d}_B \Big) \D \mathcal{V}_0 (\vr_0, \vm_0,
\vc{d}_B)\\
&\mbox{for any}\ t \geq 0, \ \Phi \in BC (X_{\mathcal{D}}),
\end{split}
\end{equation}
where $[\vr, \vm]$ is the mapping introduced in \eqref{SSM4},
satisfies
\begin{equation} \label{SM4}
\begin{split}
- &\int_0^\infty \partial_t \psi (t) \left[ \int_{X_{\mathcal{D}}}
 \Phi \left( \vr_M , \vm_M , \intO{ {E} \left( \vr, \vm \Big| \vu_B \right)
} \right)   \D \mathcal{V}_t(\vr, \vm, \vc{d}_B)  \right]  \dt \\
&+ \int_{\mathcal{D}} \left[ \int_0^\infty \psi (t) \left( \frac{\partial \Phi }{\partial e} \left( \vr_M, \vm_M, \mathcal{E}
 \right) \intO{ \mathbb{S}(\Ds \vu) : \Ds \vu } \right) \dt \right] \D \mathcal{V}_0 (\vr_0, \vm_0, \vc{d}_B)
\\
&+ \int_{\mathcal{D}} \left[ \int_0^\infty \psi (t) \left( \frac{\partial \Phi }{\partial e}
\left( \vr_M , \vm_M , \mathcal{E}
\right) \int_{\Gamma_{\rm out}} P(\vr)  \vu_B \cdot \vc{n} \ \D S_x \right) \dt \right] \D \mathcal{V}_0(\vr_0, \vm_0, \vc{d}_B)
\\
&+ \int_{\mathcal{D}} \left[ \int_0^\infty \psi (t) \left( \frac{\partial \Phi }{\partial e}
\left( \vr_M , \vm_M , \mathcal{E}
\right) \int_{\Gamma_{\rm in}} P(\vr_B)  \vu_B \cdot \vc{n} \ \D S_x \right) \dt \right] \D \mathcal{V}_0(\vr_0, \vm_0, \vc{d}_B)
\\
&\leq \psi(0) \int_{\mathcal{D}} \Phi \left( \vr_{0,M}, \vm_{0,M} , \intO{ {E} \left( \vr_0, \vm_0 \ \Big| \ \vu_B \right) } \right)
\D \mathcal{V}_0 (\vr_0, \vm_0, \vc{d}_B)\\
&+ \int_\mathcal{D} \left[ \int_0^\infty \psi(t)  \left( \sum_{i=1}^M \frac{\partial \Phi}{\partial r_i }\left( \vr_M, \vm_M,
\mathcal{E}
 \right)  \intO{ \vr \vu \cdot \Grad r_i } \right) \dt \right] \D \mathcal{V}_0 (\vr_0, \vm_0, \vc{d}_B) \\
&+ \int_{\mathcal{D}} \left[ \int_0^\infty \psi(t) \left( \sum_{i = 1}^{M}  \frac{\partial \Phi}{\partial {w}_i }\left(\vr_M, \vm_M,
\mathcal{E}
 \right) \times \right. \right. \\ & \left. \left. \times \intO{ \Big[ \vr \vu \otimes \vu : \Grad \vc{w}_i
+ p(\vr) \Div \vc{w}_i - \mathbb{S}(\Ds \vu) : \Grad \vc{w}_i + \vr \vc{g} \cdot \vc{w}_i \Big] } \right) \dt \right]
\D \mathcal{V}_0(\vr_0, \vm_0, \vc{d}_B) \\
&+ \int_{\mathcal{D}} \left[ \int_0^\infty \psi(t) \left( \frac{\partial \Phi }{\partial e} \left( \vr_M, \vm_M, \mathcal{E}  \right) \left( - \intO{ \left[ \vr \vu \otimes \vu + p(\vr) \mathbb{I} \right]  :  \Grad \vu_B }  \right.
\right. \right. \\
&- \left. \left. \left. \intO{ \vr \vu  \cdot \vu_B \cdot \Grad \vu_B  } + \intO{ \mathbb{S}(\Ds \vu) : \Ds \vu_B }   + \intO{ \vr \vc{g} \cdot (\vu - \vu_B) } \right) \right) \dt \right] \D \mathcal{V}_0(\vr_0, \vm_0, \vc{d}_B)
\end{split}
\end{equation}
for any $\psi \in C^1_c[0, \infty)$, $\psi \geq 0$, for any $M\in\N$ and any
\[
\Phi = \Phi (\vc{r},\vc{w}, e), \ \Phi \in C^1 (R^M \times R^{M} \times R),\
\nabla \Phi \in BC (R^M \times R^{M} \times R; R^M \times R^{M} \times R),  \ \frac{\partial \Phi}{\partial e } \geq 0.
\]

\end{itemize}

\end{Theorem}

\begin{Remark} \label{RST1}

For the sake of brevity, we have omitted the dependence on the data $[\vr_0, \vm_0, \vc{d}_B]$ in the
arguments of all integrals
with respect to $\D \mathcal{V}_0$ in \eqref{SM4}. In general,
\[
\int_0^\infty \psi(t) \left( \intO{ \vc{F}(\vr, \vm, \vu, \vc{d}_B) \cdot D \Phi } \right) \dt \ \D
\mathcal{V}_0[\vr_0, \vm_0, \vc{d}_B]
\]
is interpreted as
\[
\int_0^\infty \psi(t) \left( \intO{ \vc{F}(\vr (t; [\vr_0, \vm_0, \vc{d}_B]) , \vm (t; [\vr_0, \vm_0, \vc{d}_B]), \vu
(t; [\vr_0, \vm_0, \vc{d}_B]), \vc{d}_0) \cdot D \Phi } \right) \dt \ \D
\mathcal{V}_0[\vr_0, \vm_0, \vc{d}_B],
\]
where the velocity $\vu$ on any time interval $[0,T]$ is uniquely determined by $\vr$ and $\vm$, thus by
the data $[\vr_0, \vm_0, \vc{d}_B]$. The mapping
\[
(t, [\vr_0, \vm_0, \vc{d}_B]) \to \vu (t; [\vr_0, \vm_0, \vc{d}_B])
\]
must be measurable with respect to $\dt \otimes \mathcal{V}_0$ for the integral to be well defined.
The same applies to the integral containing the ``trace'' $\vr|_{\partial \Omega}$. This issue will be handled in the
proof of Theorem \ref{ST1} in Section \ref{s:ST}.

\end{Remark}

\begin{Corollary} \label{SC1}

The mapping
\[
M_t: \mathfrak{P}[\mathcal{D}] \to \mathfrak{P}[\mathcal{D}],\ M_t [\mathcal{V}_0] \mapsto
\mathcal{V}_t
\]
is a statistical solution in the dense of Definition \ref{SD1}.

\end{Corollary}

The following three sections are devoted to the proof of Theorem \ref{ST1} and Corollary \ref{SC1}. We finish this part
by some remarks on how our concept of statistical solution is related to that one introduced by Foias et al \cite{FoRoTe3}, \cite{FoRoTe2} in the context of the incompressible Navier--Stokes system.

The inequality \eqref{SM4} contains all information concerning the behavior of the parametrized measure $\{ \mathcal{V}_t \}_{t \geq 0}$. In particular, we may consider $\Phi = \Phi (\vc{r})$ independent of $\vc{w}$ and $e$. In this case, relation \eqref{SM4} reduces to equality:
\[
\begin{split}
- &\int_0^\infty \partial_t \psi (t)  \int_{X_{\mathcal{D}}} \Phi \left( \vr_M
\right)  \D \mathcal{V}_t(\vr, \vm, \vc{d}_B)    \dt
= \psi(0) \int_{\mathcal{D}} \Phi \left( \vr_{0,M}  \right)
\D \mathcal{V}_0 (\vr_0, \vm_0, \vc{d}_B)\\
&+ \int_\mathcal{D} \left[ \int_0^\infty \psi(t)  \left( \sum_{i=1}^M \frac{\partial \Phi}{\partial r_i }\left( \vr_M \right)
\intO{ \vr \vu \cdot \Grad r_i } \right) \dt \right] \D \mathcal{V}_0 (\vr_0, \vm_0, \vc{d}_B)
\end{split}
\]
that can be interpreted as the integrated version of the equation of continuity \eqref{M1}. Note that $r_i \in W^{k,2}_0(\Omega)$;
whence the boundary terms in \eqref{M1} vanish. As $\vr$ is weakly continuous in time, we may deduce
\begin{equation} \label{SM5a}
\left[ \int_{X_{\mathcal{D}}} \Phi (\vr_M ) \D \mathcal{V}_t \right]_{t = \tau_1}^{t = \tau_2}
= \int_\mathcal{D} \left[ \int_{\tau_1}^{\tau_2} \left( \sum_{i=1}^M \frac{\partial \Phi}{\partial r_i }\left( \vr_M \right)
\intO{ \vr \vu \cdot \Grad r_i } \right) \dt \right] \D \mathcal{V}_0 (\vr_0, \vm_0, \vc{d}_B)
\end{equation}
for any $0 \leq \tau_1 \leq \tau_2$.

Next, we may consider $\Phi = \Phi (\vc{w})$ obtaining an analogue of the momentum equation,
\[
\begin{split}
- &\int_0^\infty \partial_t \psi (t)  \int_{X_{\mathcal{D}}} \Phi \left( \vm_M
\right)  \D \mathcal{V}_t(\vr, \vm, \vc{d}_B)    \dt = \psi(0) \int_{\mathcal{D}} \Phi \left( \vm_{0,M}   \right)
\D \mathcal{V}_0 (\vr_0, \vm_0, \vc{d}_B)\\
&+ \int_{\mathcal{D}} \left[ \int_0^\infty \psi(t) \left( \sum_{i = 1}^{M}  \frac{\partial \Phi}{\partial {w}_i }\left(\vm_M
 \right) \times \right. \right. \\ & \left. \left. \times \intO{ \Big[ \vr \vu \otimes \vu : \Grad \vc{w}_i
+ p(\vr) \Div \vc{w}_i - \mathbb{S}(\Ds \vu) : \Grad \vc{w}_i + \vr \vc{g} \cdot \vc{w}_i \Big] } \right) \dt \right]
\D \mathcal{V}_0(\vr_0, \vm_0, \vc{d}_B);
\end{split}
\]
whence, similarly to \eqref{SM5a},
\begin{equation} \label{SM6a}
\begin{split}
&\left[ \int_{X_{\mathcal{D}}} \Phi (\vm_M) \D \mathcal{V}_t \right]_{t = \tau_1}^{t = \tau_2} = \int_{\mathcal{D}} \left[ \int_{\tau_1}^{\tau_2} \left( \sum_{i = 1}^{M}  \frac{\partial \Phi}{\partial {w}_i }\left(\vm_M
 \right) \times \right. \right. \\ & \left. \left. \times \intO{ \Big[ \vr \vu \otimes \vu : \Grad \vc{w}_i
+ p(\vr) \Div \vc{w}_i - \mathbb{S}(\Ds \vu) : \Grad \vc{w}_i + \vr \vc{g} \cdot \vc{w}_i \Big] } \right) \dt \right]
\D \mathcal{V}_0(\vr_0, \vm_0, \vc{d}_B).
\end{split}
\end{equation}
for any $0 \leq \tau_1 \leq \tau_2$.

Finally, we consider $\Phi = \Phi (e),\ \frac{\partial \Phi (e)}{\partial e} \geq 0$ obtaining the energy inequality:
\begin{equation} \label{SM7a}
\begin{split}
- &\int_0^\infty \partial_t \psi (t) \left[ \int_{X_{\mathcal{D}}} \Phi \left( \intO{ {E} \left( \vr, \vm \Big| \vu_B \right)}
\right)  \D \mathcal{V}_t(\vr, \vm, \vc{d}_B)  \right]  \dt \\
&+ \int_{\mathcal{D}} \left[ \int_0^\infty \psi (t) \left( \frac{\partial \Phi }{\partial e} \left( \mathcal{E}
 \right) \intO{ \mathbb{S}(\Ds \vu) : \Ds \vu } \right) \dt \right] \D \mathcal{V}_0 (\vr_0, \vm_0, \vc{d}_B)
\\
&+ \int_{\mathcal{D}} \left[ \int_0^\infty \psi (t) \left( \frac{\partial \Phi }{\partial e}
\left( \mathcal{E}
\right) \int_{\Gamma_{\rm out}} P(\vr)  \vu_B \cdot \vc{n} \ \D S_x \right) \dt \right] \D \mathcal{V}_0(\vr_0, \vm_0, \vc{d}_B)
\\
&+ \int_{\mathcal{D}} \left[ \int_0^\infty \psi (t) \left( \frac{\partial \Phi }{\partial e}
\left( \mathcal{E}
\right) \int_{\Gamma_{\rm in}} P(\vr_B)  \vu_B \cdot \vc{n} \ \D S_x \right) \dt \right] \D \mathcal{V}_0(\vr_0, \vm_0, \vc{d}_B)
\\
&\leq \psi(0) \int_{\mathcal{D}} \Phi \left( \intO{ {E} \left( \vr_0, \vm_0 \ \Big| \ \vu_B \right) }  \right)
\D \mathcal{V}_0 (\vr_0, \vm_0, \vc{d}_B)\\
&+ \int_{\mathcal{D}} \left[ \int_0^\infty \psi(t) \left( \frac{\partial \Phi }{\partial e} \left( \mathcal{E} \right) \left( - \intO{ \left[ \vr \vu \otimes \vu + p(\vr) \mathbb{I} \right]  :  \Grad \vu_B }  \right.
\right. \right. \\
&- \left. \left. \left. \intO{ \vr \vu  \cdot \vu_B \cdot \Grad \vu_B  } + \intO{ \mathbb{S}(\Ds \vu) : \Ds \vu_B }   + \intO{ \vr \vc{g} \cdot (\vu - \vu_B) } \right) \right) \dt \right] \D \mathcal{V}_0(\vr_0, \vm_0, \vc{d}_B).
\end{split}
\end{equation}
Similarly to \eqref{SM5a}, \eqref{SM6a}, we may deduce from \eqref{SM7a}:
\begin{equation} \label{SM7b}
\begin{split}
&\left[ \int_{X_{\mathcal{D}}} \Phi \left( \intO{ {E} \left( \vr, \vm \Big| \vu_B \right)}
\right)  \D \mathcal{V}_t(\vr, \vm, \vc{d}_B)  \right]_{t = \tau_1}^{t = \tau_2} \\
&+ \int_{\mathcal{D}} \left[ \int_{\tau_1}^{\tau_2}  \left( \frac{\partial \Phi }{\partial e} \left( \mathcal{E}
 \right) \intO{ \mathbb{S}(\Ds \vu) : \Ds \vu } \right) \dt \right] \D \mathcal{V}_0 (\vr_0, \vm_0, \vc{d}_B)
\\
&+ \int_{\mathcal{D}} \left[ \int_{\tau_1}^{\tau_2}  \left( \frac{\partial \Phi }{\partial e}
\left( \mathcal{E}
\right) \int_{\Gamma_{\rm out}} P(\vr)  \vu_B \cdot \vc{n} \ \D S_x \right) \dt \right] \D \mathcal{V}_0(\vr_0, \vm_0, \vc{d}_B)
\\
&+ \int_{\mathcal{D}} \left[ \int_{\tau_1}^{\tau_2}  \left( \frac{\partial \Phi }{\partial e}
\left( \mathcal{E}
\right) \int_{\Gamma_{\rm in}} P(\vr_B)  \vu_B \cdot \vc{n} \ \D S_x \right) \dt \right] \D \mathcal{V}_0(\vr_0, \vm_0, \vc{d}_B)
\\
&\leq - \int_{\mathcal{D}} \left[ \int_{\tau_1}^{\tau_2}  \left( \frac{\partial \Phi }{\partial e} \left( \mathcal{E} \right) \left(  \intO{ \left[ \vr \vu \otimes \vu + p(\vr) \mathbb{I} \right]  :  \Grad \vu_B }  \right.
\right. \right. \\
&+ \left. \left. \left. \intO{ \vr \vu  \cdot \vu_B \cdot \Grad \vu_B  } - \intO{ \mathbb{S}(\Ds \vu) : \Ds \vu_B }  - \intO{ \vr \vc{g} \cdot (\vu - \vu_B) } \right) \right) \dt \right] \D \mathcal{V}_0(\vr_0, \vm_0, \vc{d}_B)
\end{split}
\end{equation}
for a.a. $\tau_1 \leq \tau_2$ including $\tau_1 = 0$ - the Lebesgue point of the function
\[
t \mapsto \int_{X_{\mathcal{D}}} \Phi \left( \intO{ {E} \left( \vr, \vm \Big| \vu_B \right)}
\right)  \D \mathcal{V}_t(\tvr, \tvm, \vc{d}_B).
\]
Moreover, in view of weak lower semi--continuity of the energy, validity of the inequality \eqref{SM7b} extends to any $\tau_2 > 0$.

As the last comment, we point out that, differently from \cite{FoRoTe3}, \cite{FoRoTe2}, the statistical solutions constructed here enjoy two additional properties: on the one hand,
the semigroup property \eqref{SSM1} (in the spirit of Flandoli and Romito \cite{FlaRom}), and on the other hand, the ``consistency'' property \eqref{SSM2} with the classical notion of weak solution. In addition, as we show in Section \ref{C}, the statistical solutions
are continuous at measures supported by regular initial data.

\section{Existence theory for the Navier--Stokes system}
\label{ET}

The existence of global in time finite energy weak solutions for the initial data
\[
[\vr_0, \vm_0, \mathcal{E}_0],\ \mbox{with}\ \mathcal{E}_0 = \intO{ E \left(\vr_0, \vm_0 \Big| \vu_B \right) }
\]
has been proved by Chang, Jin, and Novotn\' y \cite{ChJiNo}. As a matter of fact, they {\bf (i)} replace the weak formulation
\eqref{M1} by a weaker stipulation
\begin{equation} \label{M1aa}
- \intO{ \vr_0 \varphi }  +
 \int_0^\infty \int_{\Gamma_{\rm in}} \varphi \vr_B \vu_B \cdot \vc{n} \ \D \ S_x =
\int_0^\infty \intO{ \Big[ \vr \partial_t \varphi + \vr \vu \cdot \Grad \varphi \Big] } \dt,
\end{equation}
for any $\varphi \in C^1_c([0,\infty) \times {\Omega} \cup \Gamma_{\rm in})$; {\bf (ii)} omit the (positive) integral
\[
\int_{\Gamma_{\rm out}} P(\vr) \vu_B \cdot \vc{n} \ \D S_x
\]
in the energy balance; {\bf (iii)} consider the energy balance \eqref{M3} only in the integrated form:
\begin{equation} \label{M3aa}
\begin{split}
&\left[ \intO{ {E} \left( \vr, \vm \Big| \vu_B \right)}
  \right]_{t = 0}^{t = \tau}
+  \int_{0}^{\tau}  \intO{ \mathbb{S}(\Ds \vu) : \Ds \vu } \dt
+ \int_{0}^{\tau} \int_{\Gamma_{\rm in}} P(\vr_B)  \vu_B \cdot \vc{n} \ \D S_x \dt
\\
&\leq -  \int_{0}^{\tau} \intO{ \left( \vr \vu \otimes \vu + p(\vr) \mathbb{I} \right)  :  \Grad \vu_B }  \dt
- \int_0^\tau \intO{ \vr \vu  \cdot \vu_B \cdot \Grad \vu_B  } \dt  \\ &+ \int_0^\tau \intO{ \mathbb{S}(\Ds \vu) : \Ds \vu_B }  +
\int_0^\tau \intO{ \vr \vc{g} \cdot (\vu - \vu_B) }  \dt
\end{split}
\end{equation}
for any $\tau \geq 0$.

Let us comment shortly on the modifications necessary to accommodate the present weak formulation in the framework of
\cite{ChJiNo}. To begin with, the equation of continuity can be obviously considered in the form \eqref{M1}, as the boundary integral
over the outflux part $\Gamma_{\rm out}$ is linear in $\vr$, cf. also Lions \cite{LI4}. As for
\[
\int_{\Gamma_{\rm out}} P(\vr) \vu_B \cdot \vc{n} \ \D S_x
\]
in the energy inequality \eqref{M3}, we note that the pressure potential $P$ is convex as the pressure is monotone in the present setting; whence it can be retained in \eqref{M3} as the integral is weakly l.s.c.

Finally, the energy inequality in the differential form \eqref{M3} requires {\it a priori} estimates rending the pressure potential
$P(\vr)$ equi--integrable in $\Omega$. As pointed out in Chang, Jin, and Novotn\' y \cite[Section 6]{ChJiNo}, this might be a delicate issue as the standard method, used in \cite{ChJiNo}, gives rise to bounds
\[
\int_0^T \int_K p(\vr) \vr^\alpha \dx \dt \leq c(T;K) \ \mbox{for some}\ \alpha > 0
\ \mbox{and any compact}\ K \subset \Omega.
\]
We claim that the equi--integrablity of the pressure up to the boundary can be established by using a suitable test function
\[
\bfphi(t,x) = \psi (t) \vc{w} (x),\ \psi \in C^1_c[0,\infty), \ \vc{w} \in W^{1,q}_0(\Omega; R^d),\ q >> 1,\
\Div \vc{w} \to \infty \ \mbox{as}\ x \to \partial \Omega,
\]
cf. e.g. Kuku\v cka \cite{Kuk}, or \cite[Proposition 6.1]{FP9} for details.

Implementing the above changes in the proof in \cite{ChJiNo}, we state the following result.

\begin{Proposition}[Global--in--time weak solution] \label{SP1}

Let $\Omega \subset R^d$, $d=2,3$ be a bounded domain of class $C^2$. Let the boundary data
\[
\vr_B \in C(\partial \Omega), \ \vu_B \in C^1(\Ov{\Omega}; R^d), \ \vc{g} \in C(\Ov{\Omega}; R^d),
\]
together with the initial data
\[
\vr_0, \ \vm_0 ,\ \intO{ E \left(\vr_0, \vm_0  \Big| \vu_B \right) } \leq \mathcal{E}_0,
\]
be given. Finally, suppose that the pressure is given by EOS \eqref{MH2}, with $\gamma > \frac{d}{2}$.

Then the problem \eqref{i1}--\eqref{i4} admits a finite energy weak solution $[\vr, \vm]$ in $[0, \infty) \times \Omega$ in the sense of
Definition \ref{MD1}.

\end{Proposition}

Consider the bases $\{ r_i \}_{i=1}^\infty$, $\{ \vc{w}_i \}_{i=1}^\infty$ introduced in \eqref{Mr1}, \eqref{Mr2}.
Revisiting the weak formulation \eqref{M1}, we obtain
\begin{equation} \label{SM6}
\frac{\D }{\dt} \intO{ \vr r_i } = \intO{ \vr \vu \cdot \Grad r_i },\ \intO{ \vr(0, \cdot) r_i } = \intO{ \vr_0 r_i },\
i=1,2,\dots
\end{equation}
Similarly, it follows from \eqref{M2} that
\begin{equation} \label{SM7}
\begin{split}
&\frac{\D }{\dt} \intO{ \vm \cdot \vc{w}_i } \\ &=
\intO{ \Big[ \vr \vu \otimes \vu : \Grad \vc{w}_i
+ p(\vr) \Div \vc{w}_i - \mathbb{S}(\Ds \vu) : \Grad \vc{w}_i + \vr \vc{g} \cdot \vc{w}_i \Big] }, \\
\intO{ \vm(0, \cdot) \cdot \vc{w}_i } &= \intO{ \vm_0 \cdot \vc{w}_i },\ i = 1,2,\dots
\end{split}
\end{equation}
Finally, the energy inequality \eqref{M3} can be interpreted as
\begin{equation} \label{SM8}
\begin{split}
&\frac{\D}{\dt} \Ecd   +
\intO{ \mathbb{S}(\Ds \vu) : \Ds \vu }
+ \int_{\Gamma_{\rm out}} P(\vr)  \vu_B \cdot \vc{n} \ \D S_x   +
\int_{\Gamma_{\rm in}} P(\vr_B)  \vu_B \cdot \vc{n} \ \D S_x
\\	
&\leq -
\intO{ \left[ \vr \vu \otimes \vu + p(\vr) \mathbb{I} \right]  :  \Grad \vu_B }   -
\intO{ {\vr} \vu  \cdot \vu_B \cdot \Grad \vu_B  }
\\ &+ \intO{ \mathbb{S}(\Ds \vu) : \Ds \vu_B }   + \intO{ \vr \vc{g} \cdot (\vu - \vu_B) },\ \ \
\Ecd(0+) \leq \mathcal{E}_0,
\end{split}
\end{equation}
as soon as $\Ecd \in BV_{\rm loc}[0, \infty)$ is considered as a c\` agl\` ad function.

Now, consider $\Phi \in C^1(R^M \times R^M \times R)$, $\Phi = \Phi(\vc{r}, \vc{w}, e)$ such that
\[
\nabla \Phi \in C_c (R^M \times R^M \times R; R^M \times R^M \times R),\ \frac{\partial \Phi}{\partial e} \geq 0.
\]
Applying the chain rule for a composition of a $C^1$ function with a BV function (see e.g. Ambrosio and Dal Maso \cite{AmbDal}),
we deduce from \eqref{SM6}--\eqref{SM8}:
\begin{equation} \label{SM44}
\begin{split}
&\frac{\D}{\dt}
 \Phi \left( \vr_M , \vm_M , \Ecd \right)
+ \frac{\partial \Phi }{\partial e} \left( \vr_M , \vm_M , \Ecd
 \right) \intO{ \mathbb{S}(\Ds \vu) : \Ds \vu }
\\&+
\frac{\partial \Phi }{\partial e}
\left( \vr_M , \vm_M , \Ecd
\right) \left(\int_{\Gamma_{\rm out}} P(\vr)  \vu_B \cdot \vc{n} \ \D S_x
+ %\frac{\partial \Phi }{\partial e}
%\left( \vr_M , \vm_M , \Ecd
%\right)
\int_{\Gamma_{\rm in}} P(\vr_B)  \vu_B \cdot \vc{n} \ \D S_x\right)
\\
&\leq \sum_{i=1}^M \frac{\partial \Phi}{\partial r_i }\left( \vr_M, \vm_M,
\Ecd
 \right)  \intO{ \vr \vu \cdot \Grad r_i }  \\
&+  \sum_{i = 1}^{M}  \frac{\partial \Phi}{\partial {w}_i }\left(\vr_M, \vm_M,
\Ecd
 \right) \intO{ \Big[ \vr \vu \otimes \vu : \Grad \vc{w}_i
+ p(\vr) \Div \vc{w}_i - \mathbb{S}(\Ds \vu) : \Grad \vc{w}_i + \vr \vc{g} \cdot \vc{w}_i \Big] }  \\
&- \frac{\partial \Phi }{\partial e} \left( \vr_M, \vm_M, \Ecd  \right) \left(\intO{ \left[ \vr \vu \otimes \vu + p(\vr) \mathbb{I} \right]  :  \Grad \vu_B }  + \intO{ \vr \vu  \cdot \vu_B \cdot \Grad \vu_B  } - \intO{ \mathbb{S}(\Ds \vu) : \Ds \vu_B }
\right)\\
&+ \frac{\partial \Phi }{\partial e} \left( \vr_M, \vm_M, \Ecd  \right) \intO{ \vr \vc{g} \cdot (\vu - \vu_B) }
\end{split}
\end{equation}
in $\mathcal{D}'(0, \infty)$, with
\begin{equation} \label{SM45}
\Phi \left( \vr_M , \vm_M , \Ecd \right)(0+) \leq
\Phi \left( \vr_{0,M} , \vm_{0,M} , \mathcal{E}_0 \right).
\end{equation}
For future use, it is more convenient to rewrite \eqref{SM44}, \eqref{SM45} in the distributional form
\begin{equation} \label{SM46}
\begin{split}
&- \int_0^\infty \frac{\D}{\dt} \psi (t)
 \Phi \left( \vr_M , \vm_M , \mathcal{E} \right) \ \dt
+ \int_0^\infty \psi (t)\frac{\partial \Phi }{\partial e} \left( \vr_M , \vm_M , \mathcal{E}
 \right) \intO{ \mathbb{S}(\Ds \vu) : \Ds \vu } \dt
\\&+
\int_0^\infty \psi (t) \frac{\partial \Phi }{\partial e}
\left( \vr_M , \vm_M , \mathcal{E}
\right) \left(\int_{\Gamma_{\rm out}} P(\vr)  \vu_B \cdot \vc{n} \ \D S_x %\dt
+ %\int_0^\infty \psi (t) \frac{\partial \Phi }{\partial e}
%\left( \vr_M , \vm_M , \mathcal{E}
%\right)
\int_{\Gamma_{\rm in}} P(\vr_B)  \vu_B \cdot \vc{n} \ \D S_x \right)\dt
\\
&\leq \psi(0) \Phi \left( \vr_{0,M} , \vm_{0,M} , \mathcal{E}_0 \right) \\
&+\int_0^\infty \psi (t)\sum_{i=1}^M \frac{\partial \Phi}{\partial r_i }\left( \vr_M, \vm_M,
\mathcal{E}
 \right)  \intO{ \vr \vu \cdot \Grad r_i } \dt  \\
&+  \int_0^\infty \psi (t)\sum_{i = 1}^{M}  \frac{\partial \Phi}{\partial {w}_i }\left(\vr_M, \vm_M,
\mathcal{E}
 \right) \intO{ \Big[ \vr \vu \otimes \vu : \Grad \vc{w}_i
+ p(\vr) \Div \vc{w}_i - \mathbb{S}(\Ds \vu) : \Grad \vc{w}_i + \vr \vc{g} \cdot \vc{w}_i \Big] } \dt \\
&- \int_0^\infty \psi (t) \frac{\partial \Phi }{\partial e} \left( \vr_M, \vm_M, \mathcal{E}  \right) \left(\intO{ \left[ \vr \vu \otimes \vu + p(\vr) \mathbb{I} \right]  :  \Grad \vu_B }  + \intO{ \vr \vu  \cdot \vu_B \cdot \Grad \vu_B  }\right) \dt \\
& + \int_0^\infty \psi (t) \frac{\partial \Phi }{\partial e} \left( \vr_M, \vm_M, \mathcal{E}  \right) \intO{ \mathbb{S}(\Ds \vu) : \Ds \vu_B }\dt \\
&+ \int_0^\infty \psi (t) \frac{\partial \Phi }{\partial e} \left( \vr_M, \vm_M, \mathcal{E}  \right) \intO{ \vr \vc{g} \cdot
(\vu - \vu_B) } \dt
\end{split}
\end{equation}
for any $\psi \in C^1_c[0, \infty)$, $\psi \geq 0$. Note that we have replaced
\[
\Ecd \ \mbox{by}\  \mathcal{E} \equiv \intO{ E\left( \vr, \vm \Big| \vu_B \right) }
\ \mbox{as}\ \Ecd(t) = \mathcal{E}(t) \ \mbox{for a.a.}\ t \in (0, \infty).
\]

\section{Measurable semiflow selection}
\label{S}

Similarly to Basari\v c \cite{Basa1}, we apply the method of Krylov \cite{KrylNV} adapted to the deterministic problems by Cardona and Kapitanski
\cite{CorKap}.

\subsection{Extended data space} \label{ss:data-space}

As the energy inequality is an indispensable part of the definition of weak solutions, it is convenient to extend the
data space $\mathcal{D}$ to include the scalar value of the total energy $\mathcal{E}_0$. Accordingly, we introduce
\[
\begin{split}
\mathcal{D}_E = &\left\{ [\vr_0, \vm_0, \mathcal{E}_0, \vr_B, \vu_B, \vc{g}]  \ \Big| \
\intO{E \left( \vr_0, \vm_0 \ \Big| \vu_B \right) } \leq \mathcal{E}_0,\
\vr_B \in C(\partial {\Omega}), \vr_B \geq \underline{\vr} > 0,
\right. \\
&\vu_B \in C^1(\Ov{\Omega}; R^d),\ \vc{g} \in C(\Ov{\Omega}; R^d)      \Big\}.
\end{split}
\]
We consider $\mathcal{D}_E$ as a closed subset of the Polish space
\[
[\vr_0, \vm_0, \mathcal{E}_0, \vr_B, \vu_B, \vc{g}] \in \widetilde{X}_{\mathcal{D}_E} \equiv L^\gamma (\Omega) \times L^{\frac{2\g}{\g+1}}(\Omega;R^d) \times R
\times C(\partial {\Omega}) \times C^1(\Ov{\Omega}; R^d) \times C(\Ov{\Omega}; R^d).
\]

The choice of the topology is rather inconsistent with the space $\mathcal{D}$ introduced in \eqref{SM2}, where the $(\vr_0,\vm_0)-$components
are considered in the large space $W^{-k,2}(\Omega)\times W^{-k,2}(\Omega;R^d)$. Note however that
\[
X_{\mathcal{D}_E} \equiv
W^{-k,2}(\Omega) \times W^{-k,2}(\Omega; R^d) \times R
\times C(\partial {\Omega}) \times C^1(\Ov{\Omega}; R^d) \times C(\Ov{\Omega}; R^d)
\]
and $\widetilde{X}_{\mathcal{D}_E}$ admit the same family of Borel sets on $\mathcal{D}_E$. Indeed the Borel sets of $\widetilde{X}_{\mathcal{D}_E}$ coincide with those of
\[
L^\gamma (\Omega)[{\rm weak}] \times L^{\frac{2\g}{\g+1}}(\Omega; R^d)[{\rm weak}] \times R
\times C(\partial {\Omega}) \times C^1(\Ov{\Omega}; R^d) \times C(\Ov{\Omega}; R^d)
\]
as $L^\gamma(\Omega)$ and $L^{\frac{2\g}{\g+1}}(\Omega;R^d)$ are reflexive separable Banach spaces. Next, we write
\[
\mathcal{D}_E = \cup_{N=1}^\infty \mathcal{D}_{E,N},\
\mathcal{D}_{E,N} = \left\{ [\vr_0, \vm_0, \mathcal{E}_0, \vr_B, \vu_B, \vc{g}] \ \Big| \ \mathcal{E}_0 \leq N \right\},
\]
where the topologies $L^\gamma (\Omega)[{\rm weak}]\times L^{\frac{2\g}{\g+1}}(\Omega;R^d)[{\rm weak}]$ and $W^{-k,2}(\Omega)\times W^{-k,2}(\Omega; R^d)$ are equivalent on
the closed sets $\mathcal{D}_{E,N}$. We may infer
that the topologies of $\widetilde{X}_{\mathcal{D}_E}$ and $X_{{\mathcal{D}_E}}$
generate the same family of Borel sets when restricted to $\mathcal{D}_E$.

\subsection{Trajectory space}

Trajectory space should accommodate the curves $t \mapsto [\vr, \vm](t, \cdot)$ as well as $t \mapsto \Ecd(t)$.
In view of the properties of the solution $[\vr, \vm]$, it is convenient to consider
\[
[\vr, \vm] \in C_{\rm loc}([0, \infty) ; W^{-k,2}(\Omega) \times W^{-k,2}(\Omega; R^d)).
\]

The total energy $\Ecd$ is defined as a c\` agl\` ad function on $[0, \infty)$, where $\Ecd(0)$ is fixed as $\mathcal{E}_0$. A suitable
function space is therefore the Skorokhod space of c\` agl\` ad functions
\[
\widehat{D}([0, \infty); R)  \subset D([-1, \infty); R),\
\widehat{D}([0, \infty); R) = \left\{ E \ \Big| \ E \ \mbox{c\` agl\` ad in} \ (0, \infty),\ E|_{[-1,0]} \equiv
\mathcal{E}_0 \geq E(0+) \right\}.
\]
The space $D([-1, \infty); R)$ endowed with a suitable metric is a Polish space.
We refer to Jakubowski \cite{Jaku} or Whitt \cite{Whitt} for the basic properties of the Skorokhod space
$D([-1, \infty); R)$.

We consider the trajectory space for the curves $t \mapsto [\vr(t,\cdot), \vm(t,\cdot), \Ecd(t)]$,
\[
\mathcal{T} = C_{\rm loc}([0, \infty) ; W^{-k,2}(\Omega))  \times C_{\rm loc}([0, \infty); W^{-k,2}(\Omega; R^d))
\times \widehat{D}([0, \infty); R).
\]
Note that $\mathcal{T}$ is a Polish space.

\subsection{General setting}

Following Cardona and Kapitanski \cite{CorKap}, we consider the abstract scheme based for general measurable mappings from the
data space $\mathcal{D}_E$ to the trajectory space $\mathcal{T}$. For each
$[\vr_0, \vm_0, \mathcal{E}_0, \vc{d}_B] \in \mathcal{D}_E$ we consider the set
\[
\begin{split}
\mathcal{U}[\vr_0, \vm_0, \mathcal{E}_0, \vc{d}_B] &\subset \mathcal{T},\\
\mathcal{U}[\vr_0, \vm_0, \mathcal{E}_0, \vc{d}_B] &= \left\{ [\vr, \vm, \Ecd] \ \Big|
\ [\vr, \vm, \Ecd] \ \mbox{is a finite energy weak solution with the data} \ [\vr_0, \vm_0, \mathcal{E}_0, \vc{d}_B] \right\}.
\end{split}
\]
Accordingly,
\[
\mathcal{U}: \mathcal{D}_E \to 2^\mathcal{T}
\]
can be considered as a multivalued mapping ranging in the subsets of the trajectory space $\mathcal{T}$.

We shall verify that $\mathcal{U}$ satisfies the following conditions:
\begin{itemize}
\item {\bf [A1] Existence.} For each $[\vr_0, \vm_0, \mathcal{E}_0, \vc{d}_B]$ there exists at least one finite energy weak solution
$[\vr, \vm, \Ecd] \in \mathcal{U}[\vr_0, \vm_0, \mathcal{E}_0, \vc{d}_B]$.
\item {\bf [A2] Compactness.}
The set $\mathcal{U}[\vr_0, \vm_0, \mathcal{E}_0, \vc{d}_B]$ is a compact subset of $\mathcal{T}$ for any
$[\vr_0, \vm_0, \mathcal{E}_0, \vc{d}_B] \in \mathcal{D}_E$.
\item {\bf [A3] Measurability.}
The mapping
\[
\mathcal{U}: \mathcal{D}_E \to 2^\mathcal{T}
\]
is Borel measurable with respect to the topology of $\widetilde{X}_{\mathcal{D}_E}$ on $\mathcal{D}_E$ and the Hausdorff complementary topology defined on compact subsets of $2^\mathcal{T}$.
\item {\bf [A4] Shift property.}
If
\[
[\vr, \vm, \Ecd] \in \mathcal{U}[\vr_0, \vm_0, \mathcal{E}_0, \vc{d}_B],
\]
then
\[
[\tvr, \tvm, \widetilde{\Ecd}], \ \mbox{defined as} \ \tvr(t, \cdot) = \vr(t + T; \cdot),\ \tvm(t, \cdot) = \vr(t + T; \cdot), \
\widetilde{\Ecd}(t) = \Ecd(t + T),\ t \geq 0,
\]
belongs  to
\[
\mathcal{U} \left[ \vr(T, \cdot), \vm(T, \cdot), \Ecd(T), \vc{d}_B \right]
\]
for any $T \geq 0$.

\item {\bf [A5] Continuation property.}

Let $[\vr, \vm, \Ecd] \in \mathcal{U}[\vr_0, \vm_0, \mathcal{E}_0, \vc{d}_B]$ and
$[\tvr, \tvm, \widetilde{\Ecd}] \in \mathcal{U}[\vr(T, \cdot), \vm(T, \cdot), \mathcal{E}(T), \vc{d}_B]$ for some $T \geq 0$.

Then
\[
[\vr, \vm, \Ecd] \cup_T [\tvr, \tvm, \widetilde{\Ecd}] \equiv
\left\{ \begin{array}{l} {[}\vr(t, \cdot), \vm (t, \cdot) , \Ecd (t) ] \ \mbox{for}\ t \leq T, \\
{[}\tvr(t - T), \tvm(t - T), \widetilde{\Ecd}(t - T)] \ \mbox{for}\ t > T
\end{array} \right.
\]
belongs to
\[
\mathcal{U}[\vr_0, \vm_0, \mathcal{E}_0, \vc{d}_0].
\]

\end{itemize}

The following result was proved by Cardona and Kapitanski \cite{CorKap}, see also \cite{BreFeiHof19}.

\begin{Proposition} \label{SSP1}

Suppose that both the data space $\mathcal{D}_E$ and the trajectory space $\mathcal{T}$ are Polish spaces. Let
$\mathcal{U}$,
\[
\mathcal{U}: \mathcal{D}_E \to 2^\mathcal{T}
\]
be a set--valued mapping satisfying the axioms {\rm [A1] -- [A5]}.

The there exists a measurable semi--flow selection $\vc{U}$ -- a mapping
\[
\vc{U}: \mathcal{D}_E \to \mathcal{T},
\]
enjoying the following properties:
\begin{itemize}
\item
\[
\vc{U}(\vr_0, \vm_0, \mathcal{E}_0, \vc{d}_B)  \in \mathcal{U}[\vr_0, \vm_0, \mathcal{E}_0, \vc{d}_B]
\ \mbox{for any}\ [\vr_0, \vm_0, \mathcal{E}_0, \vc{d}_B] \in \mathcal{D}_E.
\]
\item
\[
\vc{U}: \mathcal{D}_E \to \mathcal{T} \ \mbox{is Borel measurable}.
\]
\item
\[
\vc{U}(\vr_0, \vm_0, \mathcal{E}_0, \vc{d}_B) (t + s) =
\vc{U} \left( \vc{U}(s, \cdot)[\vr_0, \vm_0, \mathcal{E}_0, \vc{d}_B], \vc{d}_0 \right) (t) \ \mbox{for any}\ t,s \geq 0.
\]

\end{itemize}

\end{Proposition}

\begin{Remark} \label{SRP1}

As the mapping
\[
\vc{U}: \mathcal{D}_E \to \mathcal{T}
\]
is Borel measurable, we observe that
\[
\vc{U}(t): \mathcal{D}_E \to \mathcal{D}_E \ \mbox{is Borel measurable for any}\ t \geq 0,
\]
and
\[
\vc{U}: [0, \infty) \times \mathcal{D}_E \to \mathcal{D}_E
\]
is jointly Borel measurable. % with respect to $\dt \times \nu$ for any complete Borel probability measure  $\nu$ on $\mathcal{D}_E$.
Indeed for each trajectory
\[
\vc{U}: [0, \infty) \to \mathcal{T}
\]
we introduce its left regularization
\[
\vc{U}_\ep (t) = \frac{1}{\ep} \int_{t - \ep}^t \vc{U}(s) \ {\rm d}s, \ t \geq 0,
\]
where $\vc{U}$ has been extended to be constant in $[-1,0]$. The functions $\vc{U}_\ep$ being continuous are jointly
(Borel) measurable in $[0, \infty) \times \mathcal{D}_E$. By the same token, the
mapping
\[
\vc{U}_\ep (t): \mathcal{D}_E \to \mathcal{D}_E \ \mbox{is Borel measurable.}
\]
As $\vr$, $\vm$ are continuous functions of $t$ and $\Ecd$ is c\` agl\` ad, $\vc{U}_\ep$ converges to $\vc{U}$ \emph{pointwise} in $[0, \infty)
\times \mathcal{D}_E$ as $\ep \to 0$.

\end{Remark}

\bigskip

In the present setting, axiom [A1] is satisfied in view of the existence result stated in Proposition \ref{SP1}. Axioms
[A4], [A5] can be verified in the same way as in Basari\v c \cite{Basa1}, \cite{BreFeiHof19}. Finally, as observed in \cite{BreFeiHof19},
axioms [A2], [A3] follow
from the property of weak sequential stability stated below.

\begin{Proposition}[Weak sequential stability] \label{SSP2}

Let
\[
[\vr_{0,n}, \vm_{0,n}, \mathcal{E}_{0,n}, \vc{d}_{B,n}] \in \mathcal{D}_E
\to [\vr_{0}, \vm_{0}, \mathcal{E}_{0}, \vc{d}_{B}] \in \mathcal{D}_E
\]
in the topology of $\widetilde{X}_{\mathcal{D}_E}$.

Then any sequence of finite energy weak solutions $[\vr_n, \vm_n, \mathcal{E}_{\rm cg,n}]$ contains a subsequence
(not relabeled here) such that
\[
[\vr_n, \vm_n, \mathcal{E}_{\rm cg,n}] \to [\vr, \vm, \Ecd] \ \mbox{in}\ \mathcal{T},
\]
where $[\vr, \vm, \Ecd]$ is a finite energy weak solutions corresponding to the data $[\vr_{0}, \vm_{0}, \mathcal{E}_{0}, \vc{d}_{B}]$.

\end{Proposition}

\begin{Remark} \label{SRP2}

We point out that it is aboslutely necessary here to consider the topology of the space $\wtilde X_{\mc{D}_E}$, in particular, the initial densities must converge \emph{strongly} in $L^\gamma(\Omega)$.

\end{Remark}

\begin{proof}

Revisiting the proof of global existence in \cite{ChJiNo}, Proposition \ref{SSP2} basically coincides with the last step
of the existence proof - the artificial pressure limit - with the necessary modifications indicated in Section \ref{ET} above. The only
additional issue is therefore showing convergence of the energies,
\begin{equation} \label{TrS2}
\mathcal{E}_{{\rm cg,n}} \to \Ecd \ \mbox{in the Skorokhod space}\ \widehat{D}([0, \infty); R).
\end{equation}

In view of the energy inequality \eqref{M3}, the functions
\begin{equation} \label{TrS1}
\begin{split}
\mathcal{G}_n (t) &\equiv \mathcal{E}_{\rm cg,n}(t) + \int_0^t F_n (t) \dt,\\ \mbox{where} \
F_n &\equiv
\int_{\Gamma_{\rm in}} P(\vr_{B,n}) \vu_{B,n} \cdot \vc{n} \ \D S_x +
\intO{ \left[ \vr_n \vu_n \otimes \vu_n + p(\vr_n) \mathbb{I} \right]  :  \Grad \vu_{B,n} }  + \intO{ {\vr_n} \vu_n  \cdot \vu_{B,n} \cdot \Grad \vu_{B,n}  }
\\ &- \intO{ \mathbb{S}(\Ds \vu_n) : \Ds \vu_{B,n} }  - \intO{ \vr_n \vc{g}_{n} \cdot (\vu_n - \vu_{B,n}) },
\end{split}
\end{equation}
are non--increasing and bounded above by $\mathcal{E}_{0,n}$ in $[0, \infty)$. By virtue of the uniform energy bounds and the
compactness arguments, we deduce
\[
\left[ \tau \mapsto \int_0^\tau F_n(t) \dt \right] \to \left[ \tau \mapsto \int_0^\tau F(t) \dt \right] \ \mbox{in}\ C_{\rm loc}[0, \infty),
\]
where
\[
\begin{split}
F &\equiv
\int_{\Gamma_{\rm in}} P(\vr_{B}) \vu_{B} \cdot \vc{n} \ \D S_x +
\intO{ \left[ \vr \vu \otimes \vu + p(\vr) \mathbb{I} \right]  :  \Grad \vu_{B} }  + \intO{ {\vr} \vu  \cdot \vu_{B} \cdot \Grad \vu_B  }
\\ &- \intO{ \mathbb{S}(\Ds \vu) : \Ds \vu_{B} }  - \intO{ \vr \vc{g} \cdot (\vu - \vu_B) }.
\end{split}
\]

Accordingly, for \eqref{TrS2} to hold, it is enough to show
\[
\mathcal{G}_n \to \mathcal{G}\ \mbox{in}\ D([-1, \infty); R),
\]
where we have extended
\[
\mathcal{G}_n (t) = \mathcal{E}_{0,n}\ \mbox{for}\ t \in [-1,0],
\]
and where
\[
\mathcal{G}(\tau) = \left\{ \begin{array}{ll} \mathcal{E}_0 \quad & \mbox{if}\ \tau \in [-1,0],\\
\Ecd(\tau) + \int_0^\tau F(t) \dt \quad & \mbox{if}\ \tau>0. \end{array} \right.
\]
As $\mathcal{G}_n$ are non--increasing, and, obviously, converge uniformly for $t \in [-1,0]$ to $\mathcal{G}$, the convergence
in the space $D[-1, \infty; R)$ is equivalent to showing
\begin{equation} \label{TrS3}
\mathcal{G}_n(t) \to \mathcal{G}(t)\ \Leftrightarrow \
\mathcal{E}_{\rm cg,n}(t) \to \Ecd(t) \ \mbox{for a dense set of times}\ t \in (0, \infty),
\end{equation}
see Whitt \cite[Corollary 12.5.1]{Whitt}. Seeing that
\[
\Ecd(t) = \intO{ E \left( \vr, \vm \Big| \vu_B \right)(t, \cdot) } \ \mbox{for a.a.}\ t \in (0, \infty),
\]
and
\[
\intO{ E \left( \vr_n, \vm_n \Big| \vu_{B,n} \right) } \to
\intO{ E \left( \vr, \vm \Big| \vu_{B} \right) } \ \mbox{in}\ L^1_{\rm loc}[0, \infty),
\]
we may infer that \eqref{TrS3} holds for a.a. $t \in (0, \infty)$, passing to a suitable subsequence as the case may be.
\end{proof}

Consequently, we may apply the abstract result stated in Proposition \ref{SSP1} to the family of finite energy weak solutions
to the Navier--Stokes system.

\begin{Proposition}[Semiflow selection] \label{PrS1}
To each data
\[
[\vr_0, \vm_0, \mathcal{E}_0, \vc{d}_B] \in \mathcal{D}_E
\]
we can associate a finite energy weak solution of the Navier--Stokes system \eqref{i1}--\eqref{i4} in $[0, \infty) \times \Omega$,
\[
[\vr, \vm, \Ecd] \in \mathcal{T},\ [\vr, \vm, \Ecd] = [\vr, \vm, \Ecd]\Big(t; [\vr_0, \vm_0, \mathcal{E}_0, \vc{d}_B]\Big)
\]
in such a way that the following holds:
\begin{itemize}
\item
\[
[\vr, \vm, \Ecd]\Big(t + s; [\vr_0, \vm_0, \mathcal{E}_0, \vc{d}_B] \Big) =
[\vr, \vm, \Ecd]\Big( t; [\vr, \vm, \Ecd](s) [\vr_0, \vm_0, \mathcal{E}_0, \vc{d}_B], \vc{d}_B \Big)
\ \mbox{for any}\ s,t \geq 0;
\]
\item
the mapping
\[
[\vr, \vm, \Ecd](t, \cdot): \mathcal{D}_E \to \Big[ W^{-k,2}(\Omega) \times W^{-k,2}(\Omega; R^d) \times R \Big] \ \mbox{is Borel measurable}
\]
for any $t \geq 0$;
\item
the mapping
\[
[\vr, \vm, \Ecd] : [0, \infty) \times \mathcal{D}_E \to \Big[ W^{-k,2}(\Omega) \times W^{-k,2}(\Omega; R^d) \times R \Big]
\]
is jointly Borel measurable. %$\dt \times \nu$ measurable for any complete Borel measure on $\mathcal{D}_E$.

\end{itemize}

\end{Proposition}

Finally, anticipating the situation considered in Theorem \ref{ST1}, we identify the data space $\mathcal{D}$ introduced
in \eqref{SM2} with a Borel subset of $\mathcal{D}$
\[
\mathcal{D} = \left\{ [\vr_0, \vm_0, \mathcal{E}_0, \vc{d}_0] \ \Big|\
\mathcal{E}_0 = \intO{ E\left( \vr_0, \vm_0 \Big| \vu_B \right) } \right\}.
\]

Reformulating the conclusion of Proposition \ref{PrS1}, we obtain the following result that proves the first part of
Theorem \ref{ST1}.

\begin{Proposition} \label{PrS2}

Let $\mathcal{D}$ be the space of data introduced in \eqref{SM2} endowed with the topology of the Banach space $X_{\mathcal{D}}$,
\[
X_{\mathcal{D}} = W^{-k,2}(\Omega) \times W^{-k,2}(\Omega; R^d) \times C(\partial{\Omega}) \times C^1(\Ov{\Omega}; R^d) \times
C(\Ov{\Omega}; R^d).
\]

Then for each data
\[
[\vr_0, \vm_0, \vc{d}_B] \in \mathcal{D}
\]
there exists a finite energy weak
solution $[\vr, \vm]$ of the Navier--Stokes system \eqref{i1}--\eqref{i4} in $[0, \infty) \times \Omega$
such that the mapping
\[
[\vr, \vm]: [0, \infty) \times \mathcal{D} \to
\mathcal{D},\ [\vr, \vm] = [\vr, \vm]\Big(t; [\vr_0, \vm_0, \vc{d}_B]\Big),
\]
enjoys the following properties:
\begin{itemize}
\item
for any $[\vr_0, \vm_0, \vc{d}_B] \in \mathcal{D}$, there exists a set of times $\mathcal{R} \subset [0, \infty)$,
\[
0 \in \mathcal{R},\ \left| [0, \infty) \setminus \mathcal{R} \right| = 0,
\]
and
\[
[\vr, \vm]\Big(t + s; [\vr_0, \vm_0, \vc{d}_B] \Big) =
[\vr, \vm]\Big( t; [\vr, \vm](s) [\vr_0, \vm_0, \mathcal{E}_0, \vc{d}_B], \vc{d}_B \Big)
\ \mbox{for any}\ t \geq 0,\ s \in \mathcal{R};
\]
\item
the mapping
\[
[\vr, \vm](t, \cdot): \mathcal{D} \to \Big[ W^{-k,2}(\Omega) \times W^{-k,2}(\Omega; R^d)\Big] \ \mbox{is Borel measurable}
\]
for any $t \geq 0$;
\item
the mapping
\[
[\vr, \vm] : [0, \infty) \times \mathcal{D} \to \Big[ W^{-k,2}(\Omega) \times W^{-k,2}(\Omega; R^d) \Big]
\]
is Borel measurable in $[0, \infty) \times \mathcal{D}$.

\end{itemize}

\end{Proposition}

\section{Statistical solutions - proof of Theorem \ref{ST1}}
\label{s:ST}

Our ultimate goal is to complete the proof of Theorem \ref{ST1} and to show Corollary \ref{SC1}. The existence of
the mapping $[\vr, \vm]$ satisfying \eqref{SSM4}--\eqref{SSM4c}
has been established in Proposition \ref{PrS2}.

In accordance with \eqref{SSM5}, we set
\[
\int_{X_{\mathcal{D}}} \Phi \Big( \vr, \vm, \vc{d}_B \Big) \ \D \mathcal{V}_t(\vr, \vm, \vc{d}_B) \equiv
\int_{\mathcal{D}} \Phi \Big( \vr(t; \vr_0, \vm_0, \vc{d}_0) , \vm (t; \vr_0, \vm_0, \vc{d}_B), \vc{d}_B \Big)
\D \mathcal{V}_0 (\vr_0, \vm_0, \vc{d}_B),
\]
for any $\Phi \in BC(\mathcal{D})$,
meaning $\mathcal{V}_t$ is the pushforward measure associated to the mapping $[\vr, \vm](t; \cdot)$. The desired relation
\eqref{SM4} is then obtained by integrating \eqref{SM46} with respect to $\mathcal{V}_0$. Note that the boundary trace
of the density $\vr|_{\partial \Omega}$ as well as the velocity field $\vu$ are uniquely determined as distributions
in $(0, \infty) \times [\partial \Omega, \Omega]$ in terms of $[\vr, \vm, \vc{d}_B]$, namely
\[
\int_0^\infty \int_{\partial \Omega} \varphi \; \vr|_{\partial \Omega}\; \vu_B \cdot \vc{n} \ \D \ S_x
 =
\int_0^\infty \intO{ \Big[ \vr \partial_t \varphi + \vm \cdot \Grad \varphi \Big] } \dt
\]
for any $\varphi \in C^1_c((0, \infty) \times \oline{\Omega})$, and
\[
\begin{split}
&\int_0^\infty \intO{ \mathbb{S}(\Ds \vu) : \Grad \bfphi } \\ &=
\int_0^\infty \intO{ \Big[ \vm \cdot \partial_t \bfphi + \frac{\vm \otimes \vm}{\vr} : \Grad \bfphi
+ p(\vr) \Div \bfphi + \vr \vc{g} \cdot \bfphi \Big] },\ \vu|_{\partial \Omega} = \vu_B,
\end{split}
\]
for any $\bfphi \in C^1_c((0,\infty) \times {\Omega};R^d)$. In particular, all arguments in the integrals in
\eqref{SM4} are $\oline{\dt \times \mathcal{V}_0}$ measurable. We have proved Theorem \ref{ST1}.

To see Corollary \ref{SC1}, we have to establish the a.a. semigroup property of the Markov operators $M_t$. To this end, we write
\[
\begin{split}
\int_0^\infty &\psi(s) \int_{\mathcal{D}} \Phi(\vr, \vm, \vc{d}_B ) \D M_{t + s}[\mc V] \D s =
\int_0^\infty \psi(s) \left[ \int_{\mathcal{D}}
\Phi \Big( [\vr, \vm] (t + s; [\vr_0, \vm_0, \vc{d}_B]),  \vc{d}_0 \Big) \D \mc V \right] \D s\\
&=\int_0^\infty \psi(s) \left[ \int_{\mathcal{D}}
\Phi \left( [\vr, \vm] \Big(t; \vr(s; \vr_0, \vm_0, \vc{d}_B), \vm (s; \vr_0, \vm_0, \vc{d}_B), \vc{d}_B \Big),
 \vc{d}_B \right) \D \mc V \right] \D s
\end{split}
\]
for any $\psi \in C_c(0, \infty)$. Consequently,
\[
\int_{\mathcal{D}} \Phi(\vr, \vm, \vc{d}_B) \D M_{t + s}[\mc V]  = \int_{\mathcal{D}}
\Phi \left( [\vr, \vm] (s; \vr_0, \vm_0, \vc{d}_B),
 \vc{d}_0\right) \D M_t (\mc V) = \int_{\mathcal{D}} \Phi (\vr, \vm, \vc{d}_B ) \D (M_t \circ M_s)[\mc V]
\]
for any $t \geq 0$ and a.a. $s \in (0, \infty)$. We have proved Corollary \ref{SC1}.

\section{Conclusion, continuity with respect to the initial data}
\label{C}

We have shown the existence of \emph{statistical solution} to the barotropic Navier--Stokes system with general in/out flux
boundary conditions. The statistical solution is a family $\{M_t\}_{t \geq 0}$ of Markov operators defined on the set
$\mathfrak{P}[\mathcal{D}]$ of probability measures on the data space $\mathcal{D}$ containing the initial and boundary data.
The family enjoys the a.a. semigroup property:
\[
M_{t + s}[\nu] = M_t \circ M_s[\nu] \ \mbox{for any}\ t \geq 0, \ \mbox{and a.a.}\ s \in [0, \infty),
\nu \in \mathfrak{P}[\mathcal{D}],
\]
where the exceptional set of times $s$ depends on $\nu$. If, in addition,
\[
\nu = \delta_{[\vr_0, \vm_0, \vc{d}_0]}
\]
then
\[
M_t[\nu] = \delta_{[\vr(t, \cdot) , \vm(t, \cdot), \vc{d}_0]} \ \mbox{for all}\ t \geq 0,
\]
where $[\vr, \vm]$ is a finite energy weak solution to the Navier--Stokes system with the data
$[\vr_0, \vm_0, \vc{d}_0]$, and
\[
\mathcal{E}_0 = \intO{ E \left(\vr_0, \vm_0 \Big| \vu_B \right) }.
\]
The semigroup property and the fact that the image of a Dirac delta is again a Dirac delta are the two main novelties of our theory of statistical solutions, with respect to the works for
incompressible Navier-Stokes equations \cite{FoRoTe3}, \cite{FoRoTe2}.

Our result is restricted to the pressure--density EOS \eqref{MH2} with $\gamma > \frac{d}{2}$. The boundary data are
time independent, however, the extension to non--autonomous problem is possible. The fact that the semigroup of Markov operators
is defined for a.a. time $s$ is related to the right--continuity of the energy -- strong (right) continuity of the weak solution-- at the
time $s$. Given the present state--of--the--art of the mathematical theory of the compressible Navier--Stokes system, strong
(right) continuity of the weak solutions remains an outstanding open problem.

\[
\nu \in \mathcal{D} \mapsto M_t [ \nu ] \in \mathcal{D} 
\]
is not (known to be) continuous. This is obviously related to the lack of information on \emph{uniqueness} of finite energy weak solutions. On the other hand, regular 
initial data are likely to give rise to unique regular solutions, cf. Matsumura and Nishida \cite{MANI1}, \cite{MANI}. In the following two sections, we discuss stability 
of regular data in the context of statistical solutions.

\subsection{Stability of strong solutions}

The finite energy weak solutions introduced in Definition \ref{MD1} enjoy the weak--strong uniqueness property,
see \cite[Theorem 6.3]{AbbFeiNov}, and also Kwon et al. \cite{KwoNovSat}. Specifically, if the initial and boundary data
\[
[\vr_0, \vm_0, \vc{d}_B],\ \vr_0 > 0 \ \mbox{uniformly in}\ \Omega, \ \mathcal{E}_0 = \intO{ E \left( \vr_0, \vm_0 \Big| \vu_B \right)
},
\]
give rise to a strong (Lipschitz) solution $[\tvr, \tvm]$ defined on $[0, T_{\rm max})$, then all finite energy solutions coincide with
$[\tvr, \tvm]$. In particular,
\[
M_t (\delta_{[\vr_0, \vm_0, \vc{d}_B]}) = \delta_{[\tvr(t, \cdot), \tvm(t, \cdot), \vc{d}_B]}
\ \mbox{for all}\ t \in [0, T_{\rm max}).
\]

\begin{Remark} \label{rL1}
As a matter of fact, \cite[Theorem 6.3]{AbbFeiNov} requires $C^1$--regularity of the strong solution as it applies to a larger
class of dissipative solutions introduced therein. It is easy to check that the result can be extended to Lipschitz solutions
as long as we deal with standard distributional weak solutions used in the present paper.
\end{Remark}

We introduce the set of regular data $\mathcal{D}_R$,
\[
\begin{split}
\mathcal{D}_R \equiv &\left\{[\vr_0, \vm_0, \vc{d}_B]\in\mc D\;\big|\quad
\mbox{there exists a solution}\
 [\vr,\vm] \in W^{1, \infty}((0,T) \times \Omega)),\ \inf_{(0,T) \times \Omega} \vr > 0, \right. \\ &\ \ \
0 \leq T < T_{\rm max}  \Big\}
\end{split}
\]
for some $T_{\rm max} = T_{\rm max}[\vr_0, \vm_0, \vc{d}_B] >0$.
Suppose now that $\nu\in\mathfrak{P}[\mathcal{D}]$ is such that ${\rm supp}[\nu] \subset \mathcal{D}_R$.
As a direct consequence of the weak--strong uniqueness principle, we have
\[
M_t \left( \nu \right) = \int_{\mathcal{D}} \delta_{[\tvr(t, \cdot),
\tvm(t, \cdot), \vc{d}_B]} \D \nu (\vr_0, \vm_0, \vc{d}_B)\ \ \mbox{for all}\ 0 \leq t
< \inf_{[\vr_0, \vm_0, \vc{d}_b] \in {\rm supp}[\nu]} T_{\rm max}[\vr_0, \vm_0, \vc{d}_B].
\]
Thus $M_t$ is uniquely determined, at least locally in time, as soon as the support of the initial measure $\nu$ is contained in the set $\mathcal{D}_R$ of the data which give rise to a smooth solution.

\subsection{Continuity property of statistical solutions} \label{ss:cont}

To discuss \emph{continuity} of a statistical solution $M_t(\nu)$ in $\nu$, we need a suitable distance on the set of data $\mathcal{D}$. To this end, 
following \cite[Section 5]{AbbFeiNov}, we introduce the relative energy,
\[
E \left( \vr, \vm \Big| \tvr, \tvm \right) \equiv \frac{1}{2} \vr \left| \frac{\vm}{\vr} - \frac{\tvm}{\tvr} \right|^2 + \Big( P(\vr) - P'(\tvr)(\vr - \tvr) - P(\tvr) \Big)
\]
together with
\[
\mathcal{E} \left( \vr, \vm \Big| \tvr, \tvm \right) \equiv \intO{ E \left( \vr, \vm \Big| \tvr, \tvm \right) }.
\]
Note that $E \left( \vr, \vm \Big| \tvr, \tvm \right)$ can be seen as \emph{Bregman divergence (distance)} associated to the convex functional
\[
E( \vr, \vm) = \left\{ \begin{array}{l} \frac{1}{2} \frac{|\vm|^2}{\vr} + P(\vr) \ \mbox{if}\ \vr > 0, \\  0 \ \mbox{if}\ \vm = 0, \vr = 0, \\ 
\infty \ \mbox{otherwise.} \end{array} \right., 
\]
meaning
\[
E \left( \vr, \vm \Big| \tvr, \tvm \right) = E(\vr , \vm ) - \left< \partial_{\vr, \vm} E(\tvr, \tvm); [\vr - \tvr; \vm - \tvm] \right> - E(\tvr, \tvm),
\]
see e.g. Sprung \cite{Sprung}. Motivated by Guo et al. \cite{GuHoLiYa}, we introduce Bregman--Wasserstein distance for measures on $\mathcal{D}$:
\begin{equation} \label{BW}
W_E (\nu_1, \nu_2 ) \equiv \inf_{ \mu \in \Pi(\nu_1; \nu_2) } \int_{ \mathcal{D} \times \mathcal{D} }  \mathcal{E} \left( \vr, \vm \Big| \tvr, \tvm \right) 
\D \mu ([\vr, \vm, \vc{d}_B; \tvr, \tvm, \vc{d}_B]), 
\end{equation}
where
\[
\Pi (\nu_1; \nu_2) \equiv \left\{ \mu \in \mathfrak{P}(\mathcal{D} \times \mathcal{D}) \ \Big| \ \pi_1( \mu) = \nu_1,\ \pi_2 (\mu) = \nu_2 \right\}.
\]
Although formally similar to the conventional Wasserstein distance, $W_E$ is obviously not symmetric. 
Its specific form is, however, very convenient as the cost functional coincides with the relative energy appearing in the relative energy inequality for the 
Navier--Stokes system, cf. \cite[Section 6]{AbbFeiNov}. As we shall see below, convergence in $W_E$ will imply convergence in a suitable Wasserstein distance.

We introduce the set of \emph{regular trajectories},
\[
\begin{split}
\mathcal{T}_{L,T} = \Big\{ &[\vr, \vm, \vc{d}_B] \ \Big| \ [\vr, \vm]
\ \mbox{is a Lipschitz solution of the Navier--Stokes system in}\ [0,T] \times \Omega,\\
&\mbox{with the boundary data}\ \vc{d}_B,\ \inf_{(0,T) \times \Omega} \vr \geq L^{-1},\
\| [\vr, \vm] \|_{W^{1, \infty}(0,T) \times \Omega; R^{d + 1})} \leq L \Big\}
\end{split}
\]
Correspondingly, we define the space of \emph{regular initial data}
\begin{align*}
\mc D_{L,T}\,:=\,\Big\{[\vr(0),\vm(0),\vc{d}_B] \ \Big|& \quad [\vr,\vm,\vc{d}_B]\in\mc T_{L,T}\Big\}\,\subset
\mathcal{D}_R \subset \mathcal{D}.
\end{align*}
%We also introduced the Polish space $\widetilde{D}$ -- the space of the data $\mathcal{D}$ with the topology of
%the space
%\[
%\widetilde{X}_D = L^\gamma(\Omega) \times L^{\frac{2 \gamma}{\gamma + 1}}(\Omega; R^d) \times C(\partial \Omega) \times
%C^1(\Ov{\Omega}; R^d) \times C(\Ov{\Omega}; R^d).
%\]

The following result can be seen as a sort of continuity property of the statistical solutions with respect to regular initial data.
For the sake of simplicity, we consider fixed (deterministic) boundary data $\widetilde{\vc{d}}_B$.
\begin{Theorem}[Continuity at regular data] \label{t:cont}
Let
\[
\widetilde{\vc{d}}_B = [\vr_B, \vu_B, \vc{g}] \in C(\partial \Omega) \times C^1(\Ov{\Omega}; R^d) \times C(\Ov{\Omega}; R^d),\  \inf_{\partial \Omega} \vr_B > 0
\]
be given data.
Let $\left\{ \nu_n \right\}_{n = 1}^\infty$, $\nu_n \in \mathfrak{P}({\mathcal{D}})$, $\nu\in
\mathfrak{P}({\mathcal{D}})$, be a family of probability measures satisfying
\begin{equation} \label{hyp:bound}
\begin{split}
{\rm supp}[\nu] &\subset\mc D_{L,T}
\ \mbox{for some}\ L,T > 0, \\ 
& \nu_n \left\{ \vc{d}_B = \widetilde{\vc{d}}_B \right\} = \nu \left\{ \vc{d}_B = \widetilde{\vc{d}}_B \right\} = 1.
\end{split}
\end{equation}
Let
\begin{equation} \label{hyp:bound1}
W_E \big(\nu_n,\nu\big)\,\longrightarrow\,0\qquad\mbox{ as }\quad n\ra \infty\,.
\end{equation}

Then
\begin{equation} \label{conclus}
\sup_{t\in[0,T]}W_E\big(M_t(\nu_n),M_t(\nu)\big)\,\longrightarrow\,0\qquad\mbox{ as }\quad n\ra \infty\,.
\end{equation}
\end{Theorem}

\begin{proof}
Let $(\vr,\vm)$ be a weak solution and $(\tvr,\tvm)$ a strong (Lipsichtz) solution of the problem \eqref{i1}--\eqref{i4}
in $(0,T) \times \Omega$ corresponding to the data $[\vr_0, \vm_0, \widetilde{\vc{d}}_B]$ and
$[\widetilde{\vr}_0, \widetilde{\vm}_0, \widetilde{\vc{d}}_B]$, respectively. In addition, suppose
$\inf_{(0,T) \times \Omega} \tvr > 0$. Accordingly, the velocity $\tvu \equiv \frac{\tvm}{\tvr}$ is well defined.

Exactly as in  \cite[Section 6]{AbbFeiNov}, we compute
\begin{align}
&\left[\intO{E\left(\vr,\vm\Big|\tvr,\tvm\right)}\right]_{t=0}^{t=\tau}+\int^\tau_0\intO{\Big(\mbb{S}(\Ds\vu)-\mbb S(\Ds\tvu)\Big):\big(\Ds\vu-\Ds\tvu\big)}\dt \label{est:relative}\\
&%\qquad\qquad\qquad\qquad\qquad\qquad\qquad
+\int^\tau_0\int_{\Gamma_{\rm out}}{\Big(P(\vr)-P'(\tvr)(\vr-\tvr)-P(\tvr)\Big)\vu_B\cdot\vn}\ \D S_x\dt \nonumber \\
&\qquad\qquad\leq-\int^\tau_0\intO{\vr \left(\frac{\vm}{\vr}-\tvu \right)\cdot \left(\frac{\vm}{\vr}-\tvu\right)\cdot\nabla_x\tvu}\dt \nonumber \\
&\qquad\qquad\qquad\qquad-\int^\tau_0\intO{\Big(p(\vr)-p'(\tvr)(\vr-\tvr)-p(\tvr)\Big)\ \Div\tvu}\dt \nonumber \\
&\qquad\qquad\qquad\qquad\qquad\qquad\qquad+\int^\tau_0\intO{\left(\frac{\vr}{\tvr}-1\right)\ \left(\frac{\vm}{\vr}-\tvu \right)\cdot\Div\mbb{S}(\Ds\tvu)}\dt.\nonumber
\end{align}
Next,
\[
\left|\int^\tau_0\intO{\vr \left(\frac{\vm}{\vr}-\tvu\right)\cdot \left(\frac{\vm}{\vr}-\tvu \right)\cdot\nabla_x\tvu}\dt\right|\,\leq\,
\frac{1}{2} \left\|\nabla_x\tvu\right\|_{L^\infty((0,T) \times \Omega; R^{d \times d})}\,\int^\tau_0\intO{E\left(\vr,\vm\Big|\tvr,\tvm\right)}\dt,
\]
and, in view of hypothesis \eqref{MH2},
\[
\begin{split}
&\left| \int^\tau_0\intO{\Big(p(\vr)-p'(\tvr)(\vr-\tvr)-p(\tvr)\Big)\ \Div\tvu}\dt \right| \\ &\leq c(\gamma) \left\|\nabla_x\tvu\right\|_{L^\infty((0,T) \times \Omega; R^{d \times d})}\,\int^\tau_0\intO{E\left(\vr,\vm\Big|\tvr,\tvm\right)}\dt.
\end{split}
\]
Finally, arguing as in \cite{AbbFeiNov}, we obtain
\begin{align*}
&\left|\int^\tau_0\intO{\left(\frac{\vr}{\tvr}-1\right)\ \left(\frac{\vm}{\vr}-\tvu \right)\cdot\Div\mbb{S}(\Ds\tvu)}\dt\right| \\
&\qquad\qquad
\leq\,c(\delta, \gamma) \left\|\Div\mbb{S}(\Ds\tvu) \right\|_{L^\infty((0,T) \times \Omega; R^d)}   \,\int^\tau_0\intO{E\left(\vr,\vm\Big|\tvr,\tvm\right)}\dt\,+\,\de\int^\tau_0\left\|\nabla_x(\vu-\tvu)\right\|^2_{L^2}\dt\,,
\end{align*}
for any $\de>0$. Thus if $\delta$ is small enough, the last integral on the right--hand side may be absorbed by the second term on the left--hand side
of \eqref{est:relative} via Korn's inequality.
We conclude applying Gr\"onwall's lemma:
\begin{align} \label{est:rel_fin}
&\intO{E\left(\vr,\vm\Big|\tvr,\tvm \right)(\tau)}+\int^\tau_0\left\|\Grad \vu-\Grad \tvu\right\|^2_{L^2}\dt\,\leq\,G\,\intO{E\left(\vr_0,\vm_0\Big|\tvr_0,\tvm_0\right)}
\end{align}
for any $0\leq\tau\leq T$, where
\[
G = G\left( T; \| \Grad \tvu \|_{L^\infty((0,T) \times \Omega; R^{d \times d})} ; \| \Div \mathbb{S}( \Ds \tvu )  \|_{L^\infty((0,T) \times \Omega; R^d)}
\right).
\]

At this point, we identify
\[
[\vr (t, \cdot), \vm(t, \cdot)] \approx M_t [\nu_n] (\delta_{[\vr_0, \vm_0, \widetilde{\vc{d}}_B ]}),\
[\tvr (t, \cdot), \tvm(t, \cdot)] \approx M_t [\nu] (\delta_{[\tvr_0, \tvm_0, \widetilde{\vc{d}}_B ]}),\ t \geq 0.
\]
As $[\tvr_0, \tvm_0, \widetilde{\vc{d}}_B ] \in \mathcal{D}_{L,T} \ \nu - \mbox{a.s.}$ we deduce from
\eqref{est:rel_fin} that
\begin{equation} \label{estimm}
\intO{E\left(\vr,\vm\Big|\tvr,\tvm \right)(\tau)} \leq c(L,T) \intO{E\left(\vr_0,\vm_0 \Big|\tvr_0,\tvm_0 \right)}
\end{equation}
for all $0 \leq \tau < T$ as soon as
\begin{equation} \label{estimm1}
[\vr_0, \vm_0, \widetilde{\vc{d}}_B] \in {\rm supp}[\nu_n],\
[\tvr_0 ,\tvm_0, \widetilde{\vc{d}}_B] \in {\rm supp}[\nu].
\end{equation}

In view of \eqref{estimm}, \eqref{estimm1}, and Disintegration Theorem, we may apply $\mu \in \Pi (\nu_n; \nu)$ to \eqref{estimm} obtaining
\begin{align}
&\int_{\mc D\times\mc D}\intO{E\left(\vr,\vm\Big|\tvr,\tvm\right)(\tau)}\,\D\mu[\vr_0,\vm_0,\tvr_0\tvm_0,\widetilde{\vc{d}}_B] \label{est:rel-meas} \\
&\qquad\qquad\qquad\qquad\qquad
\leq\,c(L,T) \,\int_{\mc D\times\mc D}\intO{E\left(\vr_0,\vm_0\Big|\tvr_0,\tvm_0\right)}\,\D\mu[\vr_0,\vm_0,\tvr_0,\tvm_0,\widetilde{\vc{d}}_B]; \nonumber
\end{align}
whence
\begin{align}
\inf_{ \mu \in \Pi(\nu_n; \nu)} &\int_{\mc D\times\mc D}\intO{E\left(\vr,\vm\Big|\tvr,\tvm\right)(\tau)}\,\D\mu[\vr_0,\vm_0,\tvr_0\tvm_0,\widetilde{\vc{d}}_B] \label{estim3} \\
&\qquad
\leq\,c(L,T) \inf_{ \mu \in \Pi(\nu_n; \nu)}\,\int_{\mc D\times\mc D}\intO{E\left(\vr_0,\vm_0\Big|\tvr_0,\tvm_0\right)}\,\D\mu[\vr_0,\vm_0,\tvr_0\tvm_0,\widetilde{\vc{d}}_B]; \nonumber
\end{align}
In accordance with \eqref{hyp:bound1},
\[
\inf_{ \mu \in \Pi(\nu_n; \nu)}\,\int_{\mc D\times\mc D}\intO{E\left(\vr_0,\vm_0\Big|\tvr_0,\tvm_0\right)}\,\D\mu[\vr_0,\vm_0,\tvr_0,\tvm_0,\widetilde{\vc{d}}_B]
\to 0 \ \mbox{as}\ n \to \infty;
\]
which implies
\begin{equation} \label{estim4}
\inf_{ \mu \in \Pi(\nu_n; \nu)} \int_{\mc D\times\mc D}\intO{E\left(\vr,\vm\Big|\tvr,\tvm\right)(\tau)}\,\D\mu[\vr_0,\vm_0,\tvr_0,\tvm_0,\widetilde{\vc{d}}_B]
\ \to 0 \ \mbox{as} \ n \to \infty
\end{equation}
uniformly for $\tau \in (0,T)$.

Finally,
\[
\begin{split}
W_E (M_t(\nu_n), M_t(\nu) ) = 
\inf_{ \widetilde{\mu} \in \Pi(M_t(\nu_n); M_t(\nu))} &\int_{\mc D\times\mc D}{ \mathcal{E}\left(\vr,\vm\Big|\tvr,\tvm\right)}\,\D
\widetilde{\mu}[\vr,\vm,\tvr, \tvm,\widetilde{\vc{d}}_B] \\ &\leq
\inf_{ \mu \in \Pi(\nu_n; \nu)} \int_{\mc D\times\mc D}\intO{E\left(\vr,\vm\Big|\tvr,\tvm\right)(\tau)}\,\D\mu[\vr_0,\vm_0,\tvr_0,\tvm_0,\widetilde{\vc{d}}_B],
\end{split}
\]
which completes the proof.
\end{proof}

\begin{Remark} \label{Wass}

Consider the space $\widetilde{\mathcal{D}}$ -- the space of data $\mathcal{D}$ endowed with the topology of 
\[
\widetilde{X}_D = L^\gamma(\Omega) \times L^{\frac{2 \gamma}{\gamma + 1}}(\Omega; R^d) \times C(\partial \Omega) \times
C^1(\Ov{\Omega}; R^d) \times C(\Ov{\Omega}; R^d).
\]
As
\[
[\tvr_0, \tvm_0, \widetilde{\vc{d}}_B] \in \mathcal{D}_{L,T},
\]
there is a (deterministic) constant $r > 0$ such that
\[
0 < r^{-1} \leq \tvr (t, \cdot) \leq r \ \nu - \mbox{a.s.}
\]
In particular,
\begin{equation} \label{estim5}
\begin{split}
E \left( \vr, \vm \Big| \tvr , \tvm \right) &\geq c(r) \left( |\vr - \tvr|^2 + |\vm - \tvm |^2 \right) \ \mbox{if}\
\frac{1}{2} r^{-1} \leq \vr \leq 2 r \\
E \left( \vr, \vm \Big| \tvr , \tvm \right) &\geq c(r) \left( 1 + \vr^\gamma + \frac{|\vm|^2}{\vr} \right) \ \mbox{otherwise}.
\end{split}
\end{equation}
It is a routine matter to show that \eqref{conclus}, together with \eqref{estim5}, imply convergence in a conventional Wasserstein norm, 
\[
\sup_{t\in[0,T]}W_q^{\widetilde{\mathcal{D}}} \big(M_t(\nu_n),M_t(\nu)\big)\,\longrightarrow\,0\qquad\mbox{ as }\quad n\ra \infty,\ 
1 \leq q < \frac{2 \gamma}{\gamma + 1}.
\]

\end{Remark}

\subsection{Maximal solutions}

The selection procedure hidden in the proof of Proposition \ref{SSP1} can be arranged in such a way that the selected
semiflow enjoys the property of \emph{maximal dissipation}. Given two finite energy weak solutions $[\vr_1, \vm_1]$,
$[\vr_2, \vm_2]$ corresponding to the same data
\[
[\vr_0, \vm_0], \ \mathcal{E}_0 = \intO{ E\left(\vr_0, \vm_0 \Big| \vu_B \right)}, \ \mbox{with the boundary data}\ \vc{d}_B,
\]
we introduce the relation $\prec$,
\[
[\vr_1, \vm_1] \prec [\vr_2, \vm_2]\ \Leftrightarrow \ \intO{ E\left(\vr_1, \vm_1 \Big| \vu_B \right) (t, \cdot) }
\leq \intO{ E\left(\vr_2, \vm_2 \Big| \vu_B \right)(t, \cdot) }
\ \mbox{for a.a.}\ t \in (0, \infty).
\]

\begin{Definition}[Maximal solution] \label{Dmax1}

Let the data
\[
[\vr_0, \vm_0], \ \mathcal{E}_0 = \intO{ E\left(\vr_0, \vm_0 \Big| \vu_B \right)}, \ \mbox{with the boundary data}\ \vc{d}_B,
\]
be given. We say that an associated finite energy weak solution $[\vr, \vm]$ is \emph{maximal} is it is minimal with respect
to the relation $\prec$.

\end{Definition}

Maximal solutions comply with the physical principle of maximal energy dissipation. It turns out that the semiflow selection
obtained in Propositions \ref{PrS1}, \ref{PrS2} can be constructed to consist of maximal solutions. A short inspection of the
proof in \cite{BreFeiHof19}, \cite{CorKap} reveals that the semiflow is constructed as a limit of successive minimization of
functionals of the type
\[
F([\vr, \vm, \Ecd]) = \int_0^\infty \exp (-\lambda t) \beta \left( [\vr, \vm, \Ecd] (t) \right) \dt,\ \lambda > 0,
\]
where $\beta \in BC(W^{-k,2}(\Omega) \times W^{-k,2}(\Omega) \times R)$. Consequently, minimizers of
\[
F([\vr, \vm, \Ecd]) = \int_0^\infty \exp (-\lambda t) {\rm arctng} \left( \Ecd (t) \right) \dt,\ \lambda > 0,
\]
are definite maximal in accordance with Definition \ref{Dmax1}.

We conclude by stating a property of maximal solutions that is of interest if the total energy is a Lyapunov function.

\begin{Theorem} \label{Tmax1}
Let the data
\[
[\vr_0, \vm_0], \ \mathcal{E}_0 = \intO{ E\left(\vr_0, \vm_0 \Big| \vu_B \right)}, \ \mbox{with the boundary data}\ \vc{d}_B,
\]
be given. Suppose that the total energy $\Ecd$ associated to \emph{any} finite energy weak solution
is non--increasing, in particular it
admits a limit
\[
\Ecd(t) \to E_\infty < \infty \ \mbox{as}\ t \to \infty.
\]

Suppose that $[\vr, \vm]$ is maximal. Then
\[
\intO{ E \left( \vr, \vm \Big| \vu_B \right) (t, \cdot) } \to E_\infty \ \mbox{as}\ t \to \infty.
\]

\end{Theorem}

\begin{proof}

In view of the inequality
\[
\Ecd(t) \geq \intO{ E \left( \vr, \vm \Big| \vu_B \right) (t, \cdot) } \ \mbox{for any}\ t \geq 0,
\]
obviously
\[
\limsup_{t \to \infty} \intO{ E \left( \vr, \vm \Big| \vu_B \right) (t, \cdot) } \leq E_\infty.
\]

Consequently, it is enough to show
\[
E_\infty \leq \intO{ E \left( \vr, \vm \Big| \vu_B \right) (T, \cdot) } \ \mbox{for any}\ T > 0.
\]
Assuming the contrary we find $T > 0$ such that
\begin{equation} \label{Mmax}
\intO{ E \left( \vr, \vm \Big| \vu_B \right) (T, \cdot) } < E_\infty \leq \Ecd(t) \ \mbox{for all}\ t \geq 0.
\end{equation}
In accordance with Proposition \ref{SP1}, we may construct a solution
$[\tvr, \tvm]$ such that
\[
[\tvr, \tvm](T, \cdot) = [\vr, \vm](T, \cdot),\ \intO{ E \left( \tvr, \tvm \Big| \vu_B \right)(T, \cdot) } =
\intO{ E \left( \vr, \vm \Big| \vu_B \right)(T, \cdot) },
\]
in particular, as the total energy is non--increasing,
\[
\intO{ E \left( \tvr, \tvm \Big| \vu_B \right)(t, \cdot) } \leq
\intO{ E \left( \vr, \vm \Big| \vu_B \right) (T, \cdot) } < \Ecd(t) \ \mbox{for all} \ t \geq T.
\]
Then, we may construct a new solution,
\[
[\widehat{\vr}, \widehat{\vm}] (t, \cdot) =
\left\{ \begin{array}{l} {[} \vr, \vm] (t, \cdot) \ \mbox{if}\ t \in [0,T], \\
{[} \tvr, \tvm] (t, \cdot) \ \mbox{if}\ t \in (T, \infty), \end{array} \right.
\]
with the property
\[
\begin{split}
\intO{ E \left( \widehat{\vr}, \widehat{\vm} \Big| \vu_B \right) (t, \cdot) } &=
\intO{ E \left( {\vr}, {\vm} \Big| \vu_B \right) (t, \cdot) },\ t \in [0,T) \\
\intO{ E \left( \widehat{\vr}, \widehat{\vm} \Big| \vu_B \right) (t, \cdot) } &<
\intO{ E \left( {\vr}, {\vm} \Big| \vu_B \right) (t, \cdot) } \ \mbox{for a.a.} \ t \in (T, \infty),
\end{split}
\]
in contrast with maximality of $[\vr, \vm]$.
\end{proof}

If $\vc{g} = \Grad G(x)$, $\vc{u}_B = 0$, it is possible to incorporate the term
\[
\intO{ \vr \vu \cdot \vc{g} } = \intO{ \vr \vu \cdot \Grad F } = \frac{{\rm d}}{{\rm d}t} \intO{\vr G}
\]
in the total energy, where the latter is indeed non--increasing in view of \eqref{M3}. More sophisticated examples when the
energy is a Lyapunov function can be constructed even for non--zero $\vu_B$ satisfying $\Ds \vu_B = 0$.
\def\cprime{$'$} \def\ocirc#1{\ifmmode\setbox0=\hbox{$#1$}\dimen0=\ht0
  \advance\dimen0 by1pt\rlap{\hbox to\wd0{\hss\raise\dimen0
  \hbox{\hskip.2em$\scriptscriptstyle\circ$}\hss}}#1\else {\accent"17 #1}\fi}

%\bibliography{citace}
%\bibliographystyle{plain}

\end{document}